\documentclass[10pt]{article}

\usepackage{amsmath,amsbsy,amssymb,latexsym,amsthm}
\usepackage{a4wide}

\def\Dbar{\leavevmode\lower.6ex\hbox to 0pt{\hskip-.23ex
    \accent"16\hss}D}

\newtheorem{theorem}{Theorem}
\newtheorem{corollary}[theorem]{Corollary}
\newtheorem{proposition}[theorem]{Proposition}
\newtheorem{definition}[theorem]{Definition}
\newtheorem{remark}[theorem]{Remark}
\newtheorem{example}[theorem]{Example}
\newtheorem{lemma}[theorem]{Lemma}

\newenvironment{keywords}{\begin{center}
\begin{minipage}[c]{13.4cm} {\bf Keywords:}} {\end{minipage}
\end{center}}

\newenvironment{msc}{\begin{center}
\begin{minipage}[c]{13.4cm} {\bf Mathematics Subject Classification 2010:}} {\end{minipage}
\end{center}}

\begin{document}

\title{Variational Calculus with Conformable Fractional Derivatives\thanks{This 
is a preprint of a paper whose final and definite form will appear 
in the \emph{IEEE/CAA Journal of Automatica Sinica}, ISSN 2329-9266. 
Submitted 01-Oct-2015; Revised 20-April-2016; Accepted 22-June-2016.}}

\author{Matheus J. Lazo$^{1}$\\
{\tt matheuslazo@furg.br}
\and
Delfim F. M. Torres$^{2,}$\thanks{Corresponding author.
Tel: +351 234370668; Fax: +351 234370066; Email: delfim@ua.pt}\\
{\tt delfim@ua.pt}}

\date{$^{1}$Instituto de Matem\'{a}tica, Estat\'{\i}stica e F\'{\i}sica
--- FURG, Rio Grande, RS, Brazil\\[0.3cm]
$^{2}$\text{Center for Research and Development in Mathematics and Applications (CIDMA)}
Department of Mathematics, University of Aveiro, 3810-193 Aveiro, Portugal}


\maketitle


\begin{abstract}
Invariant conditions for conformable fractional problems
of the calculus of variations under the presence of external
forces in the dynamics are studied. Depending on the type of
transformations considered, different necessary conditions
of invariance are obtained. As particular cases, we prove
fractional versions of Noether's symmetry theorem. Invariant conditions
for fractional optimal control problems, using the Hamiltonian formalism,
are also investigated. As an example of potential application in Physics,
we show that with conformable derivatives it is possible to
formulate an Action Principle for particles under frictional forces
that is far simpler than the one obtained with classical fractional derivatives.
\end{abstract}

\begin{msc}
26A33, 34A08, 49K05, 49K10, 49S05.
\end{msc}

\begin{keywords}
conformable fractional derivative;
Noether's theorem;
invariant variational conditions;
fractional calculus of variations;
fractional optimal control.
\end{keywords}


\section{Introduction}

Fractional calculus is a generalization of (integer) differential calculus,
allowing to define integrals and derivatives of real or complex order
\cite{Kilbas,miller,Podlubny}. It had its origin in the 1600s and
for three centuries the theory of fractional derivatives
developed as a pure theoretical field of mathematics, useful
only for mathematicians. The theory took more or less finished
form by the end of the XIX century. In the last few decades,
fractional differentiation has been ``rediscovered'' by applied
scientists, proving to be very useful in various fields:
physics (classic and quantum mechanics, thermodynamics, etc.),
chemistry, biology, economics, engineering,
signal and image processing, and control theory \cite{Machado}.
One can find in the existent literature several
definitions of fractional derivatives,
including the Riemann--Liouville, Caputo, Riesz,
Riesz--Caputo, Weyl, Grunwald--Letnikov, Hadamard, and Chen derivatives.
Recently, a simple solution to the discrepancies between known definitions
was presented with the introduction of a new fractional notion,
called the conformable derivative \cite{MR3164103}.
The new definition is a natural extension of the usual derivative,
and satisfies the main properties one expects in a derivative:
the conformable derivative of a constant is zero;
satisfies the standard formulas of the derivative of the product
and of the derivative of the quotient of two functions;
and satisfies the chain rule. Besides simple and similar
to the standard derivative, one can say that the conformable
derivative combines the best characteristics
of known fractional derivatives \cite{Abd:15}.
For this reason, the subject is now under strong development:
see \cite{And:Avery,MR3326681,MyID:337,MyID:324}
and references therein.

The fractional calculus of variation was introduced
in the context of classical mechanics when Riewe \cite{CD:Riewe:1997}
showed that a Lagrangian involving fractional time derivatives leads
to an equation of motion with non-conservative forces such as friction.
It is a remarkable result since frictional and non-conservative forces
are beyond the usual macroscopic variational treatment \cite{Bauer}.
Riewe generalized the usual calculus of variations for a Lagrangian
depending on Riemann--Liouville fractional derivatives \cite{CD:Riewe:1997}
in order to deal with linear non-conservative forces. Actually,
several approaches have been developed to generalize the calculus of
variations to include problems depending on Caputo fractional
derivatives, Riemann--Liouville fractional derivatives, Riesz
fractional derivatives and others
\cite{CD:Agrawal:2002,AT,BA,Cresson,LazoTorres1,OMT,MyID:227}
(see \cite{book:frac:ICP2,MR3331286,book:frac} for the state of the art). 
Among theses approaches, recently 
it was show that the action principle for dissipative
systems can be generalized, fixing the mathematical
inconsistencies present in the original Riewe's formulation,
by using Lagrangians depending on classical and Caputo
derivatives \cite{LazoCesar}.

In this paper we work with conformable fractional derivatives in the context
of the calculus of variations and optimal control \cite{book:frac:ICP2}. In
order to illustrate the potential application of conformable fractional
derivatives in physical problems we show that it is possible to formulate an
action principle with conformable fractional calculus for the frictional
force free from the mathematical inconsistencies found in the Riewe
original approach and far simpler than the formulations proposed in \cite{LazoCesar}.
Furthermore, we obtain a generalization of Noether's symmetry theorem for the
fractional variational problems and we also consider the conformable fractional
optimal control problem. Emmy Noether was the first who proved, in 1918, that the notions of
invariance and constant of motion are connected: when a system
is invariant under a family of transformations,
then a conserved quantity along the Euler--Lagrange extremals
can be obtained \cite{Logan,Torres2}.
All conservation laws of Mechanics, \textrm{e.g.},
conservation of energy or conservation of momentum,
are easily explained from Noether's theorem.
In this paper we study necessary conditions for invariance under
a family of continuous transformations, where the Lagrangian contains
a conformable fractional derivative of order $\alpha\in(0,1)$. When $\alpha\to1$,
we obtain some well-known results, in particular the Noether theorem \cite{Torres2}.
The advantages of our fractional results are clear.
Indeed, the classical constants of motion appear naturally
in closed systems while in practical terms closed systems do not exist:
forces that do not store energy, so-called nonconservative or
dissipative forces, are always present in real systems.
Fractional dynamics provide a good way to model
nonconservative systems \cite{CD:Riewe:1997}.
Nonconservative forces remove energy from the systems and, as a consequence,
the standard Noether constants of motion are broken \cite{Frederico:Torres3}.
Our results assert that it is still possible to obtain Noether-type
theorems, which cover both conservative and nonconservative cases,
and that this is done in a particularly simple and elegant way via the conformable
fractional approach. This is in contrast with the approaches
followed in \cite{Frederico:Torres1,Frederico:Torres2,Frederico:Torres5,Frederico:Torres6}.

The paper is organized as follows. In Section~\ref{sec:Preli}
we collect some necessary definitions and results on the conformable fractional calculus
needed in the sequel. In Section~\ref{sec:EL} we obtain the conformable fractional
Euler--Lagrange equation and in Section~\ref{sec:friction} we formulate an action principle
for dissipative systems, as an example of application and motivation
to study the conformable calculus of variations. In Section~\ref{sec:Dubois}
we present an immediate consequence of the Euler--Lagrange equation,
that we use later in Sections~\ref{sec:Invar} and \ref{sec:Noether},
where we prove, respectively, some necessary conditions
for invariant fractional problems and a conformable fractional Noether theorem.
We then review the obtained results using the Hamiltonian language
in Section~\ref{sec:Hamilton}. In Section~\ref{sec:FOC}
we consider the conformable fractional optimal control problem,
where the dynamic constraint is given by a conformable fractional derivative.
Using the Hamiltonian language, we provide an invariant condition.
In Section~\ref{sec:MultiDim} we consider the multi-dimensional case,
for several independent and dependent variables.


\section{Preliminaries}
\label{sec:Preli}

In this section we review the conformable fractional calculus
\cite{Abd:15,And:Avery,MR3164103}. The conformable fractional derivative
is a new well-behaved definition of fractional derivative,
based on a simple limit definition. We review in this section the generalization
of \cite{MR3164103} proposed in \cite{Abd:15}.

\begin{definition}
\label{Da1}
The left conformable fractional derivative of order $0< \alpha \leq 1$
starting from $a \in \mathbb{R}$ of a function $f:[a,b]\rightarrow \mathbb{R}$
is defined by
\begin{equation}
\label{a1}
\frac{d^{\alpha}_a}{dx^{\alpha}_a} f(x)=f^{(\alpha)}_a(x)
=\lim_{\epsilon \rightarrow 0}
\frac{f(x+\epsilon(x-a)^{1-\alpha})-f(x)}{\epsilon}.
\end{equation}
If the limit \eqref{a1} exist, then we say that $f$ is left $\alpha$-differentiable.
Furthermore, if $f^{(\alpha)}_a(x)$ exist for $x\in (a,b)$, then
$f^{(\alpha)}_a(a)=\lim_{x\rightarrow a^+}f^{(\alpha)}_a(x)$
and $f^{(\alpha)}_a(b)=\lim_{x\rightarrow b^-}f^{(\alpha)}_a(x)$.

The right conformable fractional derivative of order $0< \alpha \leq 1$
terminating at $b \in \mathbb{R}$ of a function
$f:[a,b]\rightarrow \mathbb{R}$ is defined by
\begin{equation}
\label{a2}
\frac{{_bd^{\alpha}}}{{_bdx^{\alpha}}} f(x)={_bf^{(\alpha)}}(x)
=-\lim_{\epsilon \rightarrow 0} \frac{f(x+\epsilon(b-x)^{1-\alpha})-f(x)}{\epsilon}.
\end{equation}
If the limit \eqref{a2} exist, then we say that $f$ is right $\alpha$-differentiable.
Furthermore, if ${_bf^{(\alpha)}}(x)$ exist for $x\in(a,b)$, then
${_bf^{(\alpha)}}(a)=\lim_{x\rightarrow a^+}{_bf^{(\alpha)}}(x)$ and
${_bf^{(\alpha)}}(b)=\lim_{x\rightarrow b^-}{_bf^{(\alpha)}}(x)$.
\end{definition}

It is important to note that for $\alpha=1$ the conformable fractional derivatives
\eqref{a1} and \eqref{a2} reduce to first order ordinary derivatives. Furthermore,
despite the definition of the conformable fractional derivatives \eqref{a1}
and \eqref{a2} can be generalized for $\alpha>1$ (see \cite{Abd:15}),
we consider only $0< \alpha \leq 1$ in the present work. Is is also important to note
that differently from the majority of definitions of fractional derivative,
including the popular Riemann--Liouville and Caputo fractional derivatives,
the fractional derivatives \eqref{a1} and \eqref{a2} are local operators and
are related to ordinary derivatives if the function is differentiable
(see Remark~\ref{Ra1}). For more on local fractional derivatives, we refer
the reader to \cite{MyID:296,MyID:320} and references therein.

\begin{remark}
\label{Ra1}
If $f\in C^1[a,b]$, then we have from \eqref{a1} that
\begin{equation}
\label{a3}
f^{(\alpha)}_a(x) =(x-a)^{1-\alpha}f'(x)
\end{equation}
and from \eqref{a2} that
\begin{equation}
\label{a4}
{_bf^{(\alpha)}}(x)  =-(b-x)^{1-\alpha}f'(x),
\end{equation}
where $f'(x)$ stands for the ordinary first order derivative of $f(x)$.
\end{remark}

From \eqref{a3} and \eqref{a4} it is easy to see that the conformable fractional
derivative of a constant is zero, differently from the Riemann--Liouville
derivative of a constant, and for the power functions $(x-a)^p$ and $(b-x)^p$
one has $\frac{d^{\alpha}_a}{dx^{\alpha}_a} (x-a)^p=p(x-a)^{p-\alpha}$
and $\frac{{_bd^\alpha}}{{_bdx^{\alpha}}} (b-x)^p=p(b-x)^{p-\alpha}$
for all $p\in \mathbb{R}$.

The most remarkable consequence of definitions \eqref{a1} and \eqref{a2} is that
the conformable fractional derivatives satisfy very simple fractional versions
of chain and product rules.

\begin{proposition}[See \cite{Abd:15,MR3164103}]
\label{Ra2}
Let $0<\alpha<1$ and $f$ and $g$
be $\alpha$-differentiable functions. Then,
\begin{itemize}
\item[(i)] $(c_1 f+c_2 g)^{(\alpha)}_a(x)=c_1f^{(\alpha)}_a(x)
+c_2g^{(\alpha)}_a(x)$ and ${_b(c_1 f+c_2 g)^{(\alpha)}}(x)
=c_1{_bf^{(\alpha)}}(x)+c_2{_bg^{(\alpha)}}(x)$ for all $c_1,c_2\in \mathbb{R}$;

\item[(ii)] $(fg)^{(\alpha)}_a(x)= f^{(\alpha)}_a(x)g(x)+ f(x)g^{(\alpha)}_a(x)$
and ${_b(fg)^{(\alpha)}}(x)= {_bf^{(\alpha)}}(x)g(x)+ f(x){_bg^{(\alpha)}}(x)$;

\item[(iii)] $\left(\frac{f}{g}\right)^{(\alpha)}_a(x)= \frac{f^{(\alpha)}_a(x)g(x)
- f(x)g^{(\alpha)}_a(x)}{g^2(x)}$ and $_b{\left(\frac{f}{g}\right)}^{(\alpha)}(x)
= \frac{{_bf^{(\alpha)}}(x)g(x) - f(x){_bg^{(\alpha)}}(x)}{g^2(x)}$;

\item[(iv)] if $g(x) \ge a$, then
$(f\circ g)^{(\alpha)}_a(x)=f^{(\alpha)}_a(g(x))g^{(\alpha)}_a(x)
(g(x)-a)^{\alpha-1}$;

\item[(v)] if $g(x) \le b$, then
${_b(f\circ g)^{(\alpha)}}(x)
={_bf^{(\alpha)}}(g(x)){_bg^{(\alpha)}}(x)(b-g(x))^{\alpha-1}$;

\item[(vi)] if $g(x) < a$, then
$(f\circ g)^{(\alpha)}_a(x)
=-{_af^{(\alpha)}}(g(x))g^{(\alpha)}_a(x)(a-g(x))^{\alpha-1}$;

\item[(vii)] if $g(x) > b$, then
${_b(f\circ g)^{(\alpha)}}(x)
=-f^{(\alpha)}_b(g(x)){_bg^{(\alpha)}}(x)(g(x)-b)^{\alpha-1}$.
\end{itemize}
\end{proposition}

The simple chain and product rules given in Proposition~\ref{Ra2}
justify the increasing interest in the study of the conformable
fractional calculus, since it enable us to investigate its potential
applications as a tool to practical modeling of complex problems
in science and engineering.

The conformable fractional integrals are defined
as follows \cite{Abd:15,MR3164103}.

\begin{definition}
\label{Da2}
The left conformable fractional integral of order $0< \alpha \leq 1$ starting
from $a \in \mathbb{R}$ of a function $f\in L^1[a,b]$ is defined by
\begin{equation}
\label{a5}
I^{\alpha}_a f(x)=\int_a^xf(u)d^\alpha_au
=\int_a^x\frac{f(u)}{(u-a)^{1-\alpha}}du
\end{equation}
and the right conformable fractional integral of order $0< \alpha \leq 1$
terminating at $b \in \mathbb{R}$ of a function $f\in L^1[a,b]$ is defined by
\begin{equation}
\label{a6}
{_bI^{\alpha}} f(x)=\int_x^bf(u){_bd^\alpha}u
=\int_x^b\frac{f(u)}{(b-u)^{1-\alpha}}du.
\end{equation}
\end{definition}

It is important to mention that the conformable fractional integrals \eqref{a5}
and \eqref{a6} differ from the traditional fractional Riemann--Liouville integrals
\cite{Kilbas,miller,Podlubny} only by a multiplicative constant. Moreover,
for $\alpha=1$, the conformable fractional integrals reduce
to ordinary first order integrals.

In addition to these definitions, in the present work we make use of the following
properties of conformable fractional derivatives and integrals.

\begin{theorem}
\label{Ta1}
Let $f\in C[a,b]$ and $0<\alpha \leq 1$. Then,
\begin{equation*}
\frac{d^{\alpha}_a}{dx^{\alpha}_a}I^{\alpha}_a f(x)=f(x)
\end{equation*}
and
\begin{equation*}
\frac{{_bd^{\alpha}}}{{_bdx^{\alpha}}}{_bI^{\alpha}} f(x)=f(x)
\end{equation*}
for all $x\in [a,b]$.
\end{theorem}

\begin{theorem}[Fundamental theorem of conformable fractional calculus]
\label{Ta2}
Let $f\in C^1[a,b]$ and $0<\alpha \leq 1$. Then,
\begin{equation*}
I^{\alpha}_af^{(\alpha)}_a(x)=f(x)-f(a)
\end{equation*}
and
\begin{equation*}
{_bI^{\alpha}}{_bf^{(\alpha)}}(x)=f(x)-f(b)
\end{equation*}
for all $x\in [a,b]$.
\end{theorem}

\begin{theorem}[Integration by parts]
\label{Ta3}
Let $f,g:[a,b]\rightarrow \mathbb{R}$ be two functions
such that $fg$ is differentiable. Then,
\begin{equation}
\label{a11}
\int_a^b f(x) g^{(\alpha)}_a(x)d^{\alpha}_ax
=f(x)g(x)|_a^b - \int_a^b g(x) f^{(\alpha)}_a(x)d^{\alpha}_ax,
\end{equation}
\begin{equation}
\label{a12}
\int_a^b f(x) {_bg^{(\alpha)}}(x){_bd^{\alpha}}x
=-f(x)g(x)|_a^b - \int_a^b g(x) {_bf^{(\alpha)}}(x){_bd^{\alpha}}x,
\end{equation}
and, if $f,g:[a,b]\rightarrow \mathbb{R}$ are differentiable functions, then
\begin{equation*}
\int_a^b f(x) g^{(\alpha)}_a(x)d^{\alpha}_ax
=f(x)g(x)|_a^b + \int_a^b g(x) {_bf^{(\alpha)}}(x){_bd^{\alpha}}x.
\end{equation*}
\end{theorem}

The proof of Theorem~\ref{Ta1} follows directly from \eqref{a3}, \eqref{a4},
\eqref{a5} and \eqref{a6} since $I^{\alpha}_a f(x)$ and ${_bI^{\alpha}} f(x)$
are differentiable. On the other hand, the fundamental theorem of the conformable
fractional calculus (Theorem~\ref{Ta2}) is a direct consequence of \eqref{a3},
\eqref{a4} and definitions \eqref{a5}
and \eqref{a6} since $f,g:[a,b]\rightarrow \mathbb{R}$
are differentiable functions. Finally, the integration by parts \eqref{a11}
and \eqref{a12} follow from Proposition~\ref{Ra2} and Theorem~\ref{Ta1}.
We also need the following result.

\begin{theorem}[Chain rule for functions of several variables]
Let $f:\mathbb{R}^N \to \mathbb{R}$ ($N\in \mathbb{N}$) be a
differentiable function in all its arguments and $y_1,\ldots,y_N:\mathbb{R}
\to \mathbb{R}$ be $\alpha$-differentiable functions. Then,
\begin{equation}
\label{a14}
\frac{d^{\alpha}_a}{dx^{\alpha}_a} f(y_1(x),\ldots,y_N(x))
=\frac{\partial f}{\partial y_1}{y_1}^{(\alpha)}_a
+\frac{\partial f}{\partial y_2}{y_2}^{(\alpha)}_a
+\cdots + \frac{\partial f}{\partial y_N}{y_N}^{(\alpha)}_a
\end{equation}
and
\begin{equation}
\label{a15}
\frac{{_bd^{\alpha}}}{{_bdx^{\alpha}}} f(y_1(x),\ldots,y_N(x))
=\frac{\partial f}{\partial y_1}{_by_1}^{(\alpha)}
+\frac{\partial f}{\partial y_2}{_by_2}^{(\alpha)}
+\cdots + \frac{\partial f}{\partial y_N}{_by_N}^{(\alpha)}.
\end{equation}
\end{theorem}

\begin{proof}
For simplicity, we prove \eqref{a14} only for $N=2$. The proofs for a general
$N$ and of \eqref{a15} are similar. From \eqref{a1}
we have for $N=2$ that
\begin{equation*}
\begin{split}
\frac{d^{\alpha}_a}{dx^{\alpha}_a} & f(y_1(x),y_2(x))\\
&=\lim_{\epsilon \rightarrow 0} \frac{f(y_1(x+\epsilon(x-a)^{1-\alpha}),
y_2(x+\epsilon(x-a)^{1-\alpha}))-f(y_1(x),y_2(x))}{\epsilon}\\
&=\lim_{\epsilon \rightarrow 0} \frac{f(y_1(x+\epsilon(x-a)^{1-\alpha}),
y_2(x+\epsilon(x-a)^{1-\alpha}))-f(y_1(x),y_2(x+\epsilon(x-a)^{1-\alpha}))}{y_1(x
\quad +\epsilon(x-a)^{1-\alpha})-y_1(x)}\\
&\quad\times\frac{y_1(x+\epsilon(x-a)^{1-\alpha})-y_1(x)}{\epsilon}\\
&\quad+\lim_{\epsilon \rightarrow 0} \frac{f(y_1(x),y_2(x+\epsilon(x-a)^{1-\alpha}))
-f(y_1(x),y_2(x))}{y_2(x+\epsilon(x-a)^{1-\alpha})-y_2(x)}
\, \frac{y_2(x+\epsilon(x-a)^{1-\alpha})-y_2(x)}{\epsilon}\\
&=\frac{\partial f}{\partial y_1}{y_1}^{(\alpha)}_a
+\frac{\partial f}{\partial y_2}{y_2}^{(\alpha)}_a,
\end{split}
\end{equation*}
since $f$ is differentiable.
\end{proof}


\section{The conformable fractional Euler--Lagrange equation}
\label{sec:EL}

Let us consider first the fractional variational integral
\begin{equation}
\label{funct}
\mathcal{J}(y)=\int_a^b L\left(x,y(x),y^{(\alpha)}_a(x)\right) \,d^\alpha_ax
\end{equation}
defined on the set of continuous functions $y:[a,b]\to\mathbb R$ such that
$y^{(\alpha)}_a$ exists on $[a,b]$, where the Lagrangian
$L=L(x,y,y^{(\alpha)}_a):[a,b]\times\mathbb{R}^2 \to\mathbb R$ is of class
$C^1$ in each of its arguments. The fundamental problem of the calculus
of variations consists in finding which functions extremize functional
\eqref{funct}. In order to obtain a necessary condition for the extremum
of \eqref{funct} we need the following Lemma.

\begin{lemma}[Fundamental Lemma for conformable calculus of variation]
\label{Flemma}
Let $M$ and $\eta$ be continuous function on $[a,b]$. If
\begin{equation}
\label{Flemma1}
\int_a^b \eta(x)M(x)d^{\alpha}_a x=0
\end{equation}
for any $\eta \in C[a,b]$ with $\eta(a)=\eta(b)=0$, then
\begin{equation}
\label{Flemma2}
M(x)=0
\end{equation}
for all $x \in [a,b]$.
\end{lemma}

\begin{proof}
We do the proof by contradiction. 
From \eqref{Flemma1} we have that
\begin{equation}
\label{Fproof1}
\int_a^b \eta(x)M(x)d^{\alpha}_a x
=\int_a^b \eta(x)\frac{M(x)}{(x-a)^{1-\alpha}}dx=0.
\end{equation}
Suppose that there exist an $x_0 \in (a,b)$ such that $M(x_0)\neq 0$. 
Without loss of generality, let us assume that $M(x_0)>0$. 
Since $M$ is continuous on $[a,b]$, there exists 
a neighborhood $N^{\delta}(x_0) \subset (a,b)$ such that
\begin{equation*}
M(x)>0 \;\;\; \mbox{for all} \;\;\; x\in N^{\delta}(x_0).
\end{equation*}
Let us choose
\begin{equation}
\label{Fproof2}
\eta(x)=\left\{\begin{array}{cc}
(x-x_0-\delta)^2(x-x_0+\delta)^2 & x\in N^{\delta}(x_0) \\
0 & x\notin N^{\delta}(x_0).
\end{array} \right.
\end{equation}
Clearly, $\eta(x)$ given by \eqref{Fproof2} is continuous 
and satisfy $\eta(a)=\eta(b)=0$. Inserting \eqref{Fproof2} 
into \eqref{Fproof1} we obtain
\begin{equation*}
\int_a^b \eta(x)M(x)d^{\alpha}_a x
=\int_{x_0-\delta}^{x_0+\delta} 
(x-x_0-\delta)^2(x-x_0+\delta)^2\frac{M(x)}{(x-a)^{1-\alpha}}dx>0,
\end{equation*}
which contradicts our hypothesis. Thus,
\begin{equation*}
\frac{M(x)}{(x-a)^{1-\alpha}}>0 \;\;\; \mbox{for all} \;\;\; x\in(a,b).
\end{equation*}
Since $(x-a)^{1-\alpha}>0$ for $x\in(a,b)$, and since $M\in C[a,b]$, we get
\begin{equation*}
M(x)=0\;\;\; \mbox{ for all } \;\;\; x\in[a,b].
\end{equation*}
The proof is complete.
\end{proof}

\begin{theorem}[The conformable fractional Euler--Lagrange equation]
\label{Tb1}
Let $\mathcal{J}$ be a functional of form \eqref{funct}
with $L\in C^1\left([a,b]\times \mathbb{R}^2\right)$,
and $0<\alpha \leq 1$. Let $y:[a,b]\rightarrow \mathbb{R}$
be a $\alpha$-differentiable function with $y(a)=y_a$ and $y(b)=y_b$,
$y_a,y_b\in \mathbb{R}$. Furthermore, let
$y \frac{\partial L}{\partial y^{(\alpha)}_a}$ be a differentiable function,
and $\frac{\partial L}{\partial y^{(\alpha)}_a}$ be $\alpha$-differentiable.
If $y$ is an extremizer of $\mathcal{J}$, then $y$ satisfies the following
fractional Euler--Lagrange equation:
\begin{equation}
\label{FracELEquation}
\frac{\partial L}{\partial y}
-\frac{d^\alpha_a}{dx^\alpha_a}\left(\frac{\partial L}{\partial y^{(\alpha)}_a}\right)=0.
\end{equation}
\end{theorem}

\begin{proof}
Let $y^{\ast}$ give an extremum to \eqref{funct}. We define a family of functions
\begin{equation}
\label{proof1}
y(x)=y^{\ast}(x)+\epsilon \eta(x),
\end{equation}
where $\epsilon$ is a constant and $\eta$ is an arbitrary $\alpha$-differentiable
function satisfying $\eta\frac{\partial L}{\partial {y^{\ast}}^{(\alpha)}_a}
\in C^1$ and the boundary conditions $\eta(a)=\eta(b)=0$ (weak variations).
From \eqref{proof1}, the boundary conditions $\eta(a)=\eta(b)=0$, and the fact
that $y^{\ast}(a)=y_a$ and $y^{\ast}(b)=y_b$, it follows that function $y$
is admissible: $y$ is $\alpha$-differentiable with $y(a)=y_a$, $y(b)=y_b$,
and $y\frac{\partial L}{\partial {y^{\ast}}^{(\alpha)}_a}$ is differentiable.
Let the Lagrangian $L$ be $C^1([a,b] \times \mathbb{R}^2)$.
Because $y^{\ast}$ is an extremizer of functional $\mathcal{J}$, the Gateaux
derivative $\delta \mathcal{J}(y^{\ast})$ needs to be identically null.
For the functional \eqref{funct},
\begin{equation*}
\label{b4}
\begin{split}
\delta \mathcal{J}(y^{\ast})&=\lim_{\epsilon\rightarrow 0}
\frac{1}{\epsilon}\left( \int_a^b L\left(x,y,y^{(\alpha)}_a\right) \,d^\alpha_ax
-\int_a^b L\left(x,y^\ast,{y^\ast}^{(\alpha)}_a\right) \,d^\alpha_ax\right)\\
&=\int_a^b \left(\eta(x)\frac{\partial
L\left(x,y^\ast,{y^\ast}^{(\alpha)}_a\right)}{\partial y^{\ast}}
+\eta^{(\alpha)}_a(x)\frac{\partial
L\left(x,y^\ast,{y^\ast}^{(\alpha)}_a\right)}{\partial {y^\ast}^{(\alpha)}_a} \right)
d^\alpha_ax=0.
\end{split}
\end{equation*}
Using the integration by parts formula \eqref{a11} 
($\eta\frac{\partial L}{\partial {y^\ast}^{(\alpha)}_a}$ 
is differentiable) we get
\begin{equation}
\label{b4b}
\begin{split}
\delta \mathcal{J}(y^{\ast})
&= \int_a^b \eta(x)\left(\frac{\partial
L\left(x,y^\ast,{y^\ast}^{(\alpha)}_a\right)}{\partial y^{\ast}}
-\frac{d^\alpha_a}{dx^\alpha_a}\frac{\partial
L\left(x,y^\ast,{y^\ast}^{(\alpha)}_a\right)}{\partial
{y^\ast}^{(\alpha)}_a} \right)d^\alpha_ax=0,
\end{split}
\end{equation}
since $\eta(a)=\eta(b)=0$. The fractional Euler--Lagrange equation
\eqref{FracELEquation} follows from \eqref{b4b} by using the
fundamental Lemma \ref{Flemma}.
\end{proof}

\begin{definition}
A continuous function $y$ solution of \eqref{FracELEquation}
is said to be an extremal of \eqref{funct}.
\end{definition}

\begin{remark}
For $\alpha=1$, the functional $\mathcal{J}$ given by \eqref{funct}
reduces to the classical variational functional
$$
\mathcal{J}(y)=\int_0^1 L\left(x,y(x),y'(x)\right) dx
$$
and the associated Euler--Lagrange equation \eqref{FracELEquation} is
\begin{equation}
\label{ELEquation}
\frac{\partial L}{\partial y}
-\frac{d}{dx}\left(\frac{\partial L}{\partial y'}\right)=0.
\end{equation}
\end{remark}

Let us consider now the more general case where the Lagrangian 
depends on both integer order and fractional order derivatives. 
In this case the following theorem holds.

\begin{theorem}[The generalized conformable fractional Euler--Lagrange equation]
\label{Tc1}
Let $\mathcal{J}$ be a functional of form
\begin{equation}
\label{GJ}
\mathcal{J}(y)=\int_a^b L\left(x,y(x),y'(x),y^{(\alpha)}_a(x)\right)dx,
\end{equation}
with $L\in C^1\left([a,b]\times \mathbb{R}^3\right)$,
and $0<\alpha \leq 1$. Let $y:[a,b]\rightarrow \mathbb{R}$
be a differentiable function with $y(a)=y_a$ and $y(b)=y_b$,
$y_a,y_b\in \mathbb{R}$. If $y$ is an extremizer of $\mathcal{J}$,
then $y$ satisfies the following fractional Euler--Lagrange equation:
\begin{equation}
\label{GEL}
\frac{\partial L}{\partial y}-\frac{d}{dx}\left(\frac{\partial L}{\partial y'}\right)
-\frac{1}{(x-a)^{1-\alpha}}\frac{d^\alpha_a}{dx^\alpha_a}
\left(\frac{\partial \tilde{L}}{\partial y^{(\alpha)}_a}\right)=0,
\end{equation}
where $\tilde{L}\left(x,y,y',y^{(\alpha)}_a\right)
=(x-a)^{1-\alpha}L\left(x,y,y',y^{(\alpha)}_a\right)$.
\end{theorem}

\begin{proof}
Let $y^{\ast}$ give an extremum to \eqref{GJ}. We define a family of
functions as in \eqref{proof1} but with $y\in C^1[a,b]$. From \eqref{proof1}
and the boundary conditions $\eta(a)=\eta(b)=0$, and the fact
that $y^{\ast}(a)=y_a$ and $y^{\ast}(b)=y_b$, it follows that function $y$
is admissible. Because $y^{\ast}$ is an extremizer of $\mathcal{J}$, the
Gateaux derivative $\delta \mathcal{J}(y^{\ast})$ needs to be identically null.
For the functional \eqref{GJ} we have
\begin{equation*}
\begin{split}
\delta \mathcal{J}(y^{\ast})&=\lim_{\epsilon\rightarrow 0}
\frac{1}{\epsilon}\left( \int_a^b L\left(x,y,y',y^{(\alpha)}_a\right) \,dx
-\int_a^b L\left(x,y^\ast,y'^\ast,{y^\ast}^{(\alpha)}_a\right) \,dx\right)\\
&=\int_a^b \left(\eta(x)\frac{\partial
L\left(x,y^\ast,y'^\ast,{y^\ast}^{(\alpha)}_a\right)}{\partial y^{\ast}}
+\eta'(x)\frac{\partial
L\left(x,y^\ast,y'^\ast,{y^\ast}^{(\alpha)}_a\right)}{\partial y'^{\ast}}\right)dx\\
&+\int_a^b\eta^{(\alpha)}_a(x)\frac{\partial
L\left(x,y^\ast,y'^\ast,{y^\ast}^{(\alpha)}_a\right)}{\partial {y^\ast}^{(\alpha)}_a}dx\\
&=\int_a^b\eta(x) \left(\frac{\partial
L\left(x,y^\ast,y'^\ast,{y^\ast}^{(\alpha)}_a\right)}{\partial y^{\ast}}
-\frac{d}{dx}\frac{\partial
L\left(x,y^\ast,y'^\ast,{y^\ast}^{(\alpha)}_a\right)}{\partial y'^{\ast}}\right)dx\\
&+\int_a^b\eta^{(\alpha)}_a(x)\frac{\partial
\tilde{L}\left(x,y^\ast,y'^\ast,{y^\ast}^{(\alpha)}_a\right)}{\partial {y^\ast}^{(\alpha)}_a}d^\alpha_ax=0,
\end{split}
\end{equation*}
where we performed an integration by parts in the second term in the first
integral (since $\eta(a)=\eta(b)=0$), and we rewrote the second integral 
as a conformable integral by using definition \eqref{a5}. 
Using the integration by parts formula \eqref{a11}
($\eta\frac{\partial L}{\partial {y^\ast}^{(\alpha)}_a}$ is differentiable) we get
\begin{equation}
\label{proofT2}
\begin{split}
\delta \mathcal{J}(y^{\ast})&=\int_a^b\eta(x) \left(\frac{\partial
L\left(x,y^\ast,y'^\ast,{y^\ast}^{(\alpha)}_a\right)}{\partial y^{\ast}}
-\frac{d}{dx}\frac{\partial
L\left(x,y^\ast,y'^\ast,{y^\ast}^{(\alpha)}_a\right)}{\partial y'^{\ast}}\right)dx\\
&-\int_a^b\eta(x)\frac{d^\alpha_a}{dx^\alpha_a}\frac{\partial
\tilde{L}\left(x,y^\ast,y'^\ast,{y^\ast}^{(\alpha)}_a\right)}{\partial {y^\ast}^{(\alpha)}_a}d^\alpha_ax\\
&=\int_a^b\eta(x) \left((x-a)^{1-\alpha}\frac{\partial
L\left(x,y^\ast,y'^\ast,{y^\ast}^{(\alpha)}_a\right)}{\partial y^{\ast}}
-(x-a)^{1-\alpha}\frac{d}{dx}\frac{\partial
L\left(x,y^\ast,y'^\ast,{y^\ast}^{(\alpha)}_a\right)}{\partial y'^{\ast}}\right.\\
&\;\;\;\;\;\;\;\;\;\;\;\;\;\;\;\;\;\;\;\;\;\;-\left.\frac{d^\alpha_a}{dx^\alpha_a}\frac{\partial
\tilde{L}\left(x,y^\ast,y'^\ast,{y^\ast}^{(\alpha)}_a\right)}{\partial {y^\ast}^{(\alpha)}_a}
\right)d^\alpha_ax=0,
\end{split}
\end{equation}
since $\eta(a)=\eta(b)=0$. The fractional Euler--Lagrange equation
\eqref{GEL} follows from \eqref{proofT2} by using the
fundamental Lemma~\ref{Flemma}.
\end{proof}


\section{Lagrangian formulation for frictional forces}
\label{sec:friction}

As an example of potential application of the variational calculus with
conformable fractional derivatives, we formulate an action principle
for dissipative systems free from the mathematical inconsistencies found
in the Riewe approach \cite{LazoCesar} and far simpler than the
formulation proposed in \cite{LazoCesar}. The action principle we
propose states that the equation of motion for dissipative systems
is obtained by taking the limit $a\rightarrow b$ in the extremal
of the action
\begin{equation}
\label{b11}
S=\int_a^b {L}\left(x,x',x^{(\alpha)}_a\right) dt
\end{equation}
that satisfy the fractional Euler--Lagrange equation (see \eqref{GEL})
\begin{equation}
\label{b12}
\frac{\partial L}{\partial x}- \frac{d}{dt}\frac{\partial L}{\partial x'}
-\frac{1}{(t-a)^{1-\alpha}}\frac{d^\alpha_a}{dt^\alpha_a} 
\frac{\partial \tilde{L}}{\partial x^{(\alpha)}_a}=0,
\end{equation}
where $\tilde{L}\left(x,x',x^{(\alpha)}_a\right)
=(t-a)^{1-\alpha}L\left(x,x',x^{(\alpha)}_a\right)$, 
$x(t)$ is the path of the particle and $t$ is the time. 
It is important to emphasize that the condition $a\rightarrow b$ 
(also considered in the original Riewe's approach) applied 
to the action principle does not imply any restrictions 
for conservative systems, since in this case $x(t)$ is the action's extremal 
for any time interval $[a,b]$, even when $a\rightarrow b$. 
Furthermore, our action principle is simpler than the formulation 
in \cite{LazoCesar} and free from the mathematical inconsistencies 
present in Riewe's approach (see \cite{LazoCesar} for a detailed discussion). 
In order to show that our method provides us with physical Lagrangians, 
let us consider the simple problem of a particle under a frictional 
force proportional to velocity. A quadratic Lagrangian for a particle 
under a frictional force proportional to the velocity is given by
\begin{equation}
\label{b13}
L\left(x,x',x^{(\frac{1}{2})}_a\right)
=\frac{1}{2}m\left(x'\right)^2-U(x)
+\frac{\gamma}{2}\left(x^{(\frac{1}{2})}_a\right)^2,
\end{equation}
where the three terms in \eqref{b13} represent the kinetic energy, 
potential energy, and the fractional linear friction energy, respectively. 
Note that differently from Riewe's Lagrangian \cite{CD:Riewe:1997}, 
our Lagrangian \eqref{b13} is a real function with a linear friction 
energy, which is physically meaningful. Since the equation of motion 
is obtained in the limit $a\rightarrow b$, if we consider the last 
term in \eqref{b13} up to first order in $\Delta t=t-a$, we get
\begin{equation*}
\frac{\gamma}{2}\left(x^{(\frac{1}{2})}_a\right)^2 
= \frac{\gamma}{2}\left(x' \Delta t^{\frac{1}{2}}\right)^2 
\approx \frac{\gamma}{2} x' \Delta x,
\end{equation*}
that coincides, apart from the multiplicative constant $1/2$, 
with the work from the frictional force $\gamma x'$ 
in the displacement $\Delta x \approx x'\Delta t$. 
The appearance of an additional multiplicative constant 
is a consequence of the use of fractional derivatives in the Lagrangian 
and does not appear in the equation of motion after 
we apply the action principle \cite{LazoCesar}. 

\begin{remark}
It is important to stress that the order of the fractional derivative 
should be fixed to $\alpha=1/2$ in order to obtain, by a fractional 
Lagrangian, a correct equation of motion of a dissipative system. 
For $\alpha$ different from $1/2$, the Lagrangian does not describe 
a frictional system under a frictional force proportional 
to the velocity. Consequently, the fractional linear friction energy 
makes sense only for $\alpha=1/2$.
\end{remark}

The Lagrangian \eqref{b13} is physical in the sense it provides 
physically meaningful relations for the momentum and the Hamiltonian. 
If we define the canonical variables
\begin{equation*}
q_{1}=x', \;\;\; q_{\frac{1}{2}}=x^{(\frac{1}{2})}_a
\end{equation*}
and
\begin{equation*}
p_1=\frac{\partial L}{\partial q_1}=mx', 
\;\;\; p_{\frac{1}{2}}=\frac{\partial L}{\partial q_{\frac{1}{2}}}
=\gamma x^{(\frac{1}{2})}_a,
\end{equation*}
we obtain the Hamiltonian
\begin{equation}
\label{b17}
H=q_1p_1+q_{\frac{1}{2}}p_{\frac{1}{2}}-L
=\frac{1}{2}m\left(x'\right)^2+U(x)+\frac{\gamma}{2}\left(x^{(\frac{1}{2})}_a\right)^2.
\end{equation}
From \eqref{b17} we can see that the Lagrangian \eqref{b13} is physical 
in the sense it provides us a correct relation for the momentum $p_1=m\dot{x}$, 
and a physically meaningful Hamiltonian (it is the sum of all energies). 
Furthermore, the additional fractional momentum $p_{\frac{1}{2}}=\gamma x^{(\frac{1}{2})}_a$ 
goes to zero when we take the limit $a\rightarrow b$, since $x\in C^2[a,b]$.

Finally, the equation of motion for the particle is obtained 
by inserting our Lagrangian \eqref{b13} 
into the Euler--Lagrange equation \eqref{b12},
\begin{equation}
\label{b15}
m x''+\gamma (t-a)^{-\frac{1}{2}}\frac{d^{\frac{1}{2}}_a}{dt^{\frac{1}{2}}_a}\left[
(t-a)^{\frac{1}{2}}x^{(\frac{1}{2})}_a \right]
=m x''+\gamma x' + \gamma (t-a) x''=F(x),
\end{equation}
where we have used \eqref{a3} since $x\in C^2[a,b]$ 
and $F(x)=-\frac{d}{dx}U(x)$ is the external force. 
By taking the limit $a\rightarrow b$ with $t\in [a,b]$, 
we finally obtain the correct equation of motion 
for a particle under a frictional force:
\begin{equation*}
m x''+\gamma x'=F(x).
\end{equation*}


\section{The conformable fractional DuBois--Reymond condition}
\label{sec:Dubois}

In the remainder of the present work we are going to consider only the simplest case
where we have no mixed integer and fractional derivatives.
We now present the DuBois--Reymond condition in the conformable fractional context.
It is an immediate consequence of the chain rule \eqref{a14}
and the Euler--Lagrange equation \eqref{FracELEquation}.

\begin{theorem}[The conformable fractional DuBois--Reymond condition]
\label{dubois}
If $y$ is an extremal of $\mathcal{J}$ as in \eqref{funct}, then
\begin{equation}
\label{duboiscondition}
\frac{d^\alpha_a}{dx^\alpha_a}\left(
L-\frac{\partial L}{\partial y^{(\alpha)}_a} y^{(\alpha)}_a  \right)
=\frac{\partial L}{\partial x}\cdot (x-a)^{1-\alpha}.
\end{equation}
\end{theorem}

\begin{proof}
By the chain rule \eqref{a14} and the Leibniz rule in Proposition~\ref{Ra2},
\begin{equation*}
\begin{split}
\frac{d^\alpha_a}{dx^\alpha_a}
&\left( L-\frac{\partial L}{\partial y^{(\alpha)}_a} y^{(\alpha)}_a\right)\\
&=\frac{\partial L}{\partial x}x^{(\alpha)}_a
+\frac{\partial L}{\partial y}y^{(\alpha)}_a
+\frac{\partial L}{\partial y^{(\alpha)}_a}
\frac{d^\alpha_a}{dx^\alpha_a}y^{(\alpha)}_a
-\frac{d^\alpha_a}{dx^\alpha_a}\left(
\frac{\partial L}{\partial y^{(\alpha)}_a}\right) y^{(\alpha)}_a
-\frac{\partial L}{\partial y^{(\alpha)}_a}
\frac{d^\alpha_a}{dx^\alpha_a}y^{(\alpha)}_a\\
&=\frac{\partial L}{\partial x}x^{(\alpha)}_a
+y^{(\alpha)}_a\left[ \frac{\partial L}{\partial y}
-\frac{d^\alpha_a}{dx^\alpha_a}\left(
\frac{\partial L}{\partial y^{(\alpha)}_a}\right)\right]
=\frac{\partial L}{\partial x}\cdot (x-a)^{1-\alpha}.
\end{split}
\end{equation*}
The proof is complete.
\end{proof}

\begin{corollary}
If \eqref{funct} is autonomous, that is,
if $L=L(y,y^{(\alpha)}_a)$ does not depend on $x$, then
$$
\frac{d^\alpha_a}{dx^\alpha_a}\left(
L-\frac{\partial L}{\partial y^{(\alpha)}_a} y^{(\alpha)}_a  \right)=0
$$
along any extremal $y$.
\end{corollary}

\begin{remark}
When $\alpha = 1$ and $y\in C^1$, Theorem~\ref{dubois}
is the classical DuBois--Reymond condition: if $y\in C^1$
is an extremal of $\mathcal{J}(y)=\int_0^1L(x,y,y') dx$
(\textrm{i.e.}, $y$ satisfies \eqref{ELEquation}), then
$$
\frac{d}{dx}\left( L-\frac{\partial L}{\partial y'} y'\right)
=\frac{\partial L}{\partial x}.
$$
\end{remark}


\section{Fractional invariant conditions}
\label{sec:Invar}

We consider invariance transformations in the $(x,y)$-space,
depending on a real parameter $\epsilon$. To be more precise,
we consider transformations of type
\begin{equation}
\label{trans}
\left\{
\begin{array}{l}
\overline x=x+\epsilon \tau (x,y(x)),\\
\overline y=y+\epsilon \xi (x,y(x)),\\
\end{array}
\right.\end{equation}
where the generators $\tau$ and $\xi$ are such that
$\overline x \ge a$ and there exist
$\tau^{(\alpha)}_a$ and $\xi^{(\alpha)}_a$.

\begin{definition}
\label{invariantDef}
We say that the fractional variational integral \eqref{funct} is invariant under
the family of transformations \eqref{trans} up to the Gauge term $\Lambda$,
if a function $\Lambda=\Lambda(x,y)$ exists such that for any function $y$
and for any real $x\in[a,b]$, we have
\begin{equation}
\label{invar}
L\left(\overline x,\overline y,
\frac{d^\alpha_a\overline y}{d\overline x^\alpha_a}\right)
\frac{d^\alpha_a\overline x}{d^\alpha_a x}
=L(x,y,y^{(\alpha)}_a)+\epsilon \frac{d^\alpha_a\Lambda}{dx^\alpha_a}(x,y)
+ o(\epsilon)
\end{equation}
for all $\epsilon$ in some neighborhood of zero,
where $\frac{d^\alpha_a\overline x}{d^\alpha_a x}$ stands for
\begin{equation}
\label{eq:standsfor}
\frac{\frac{d^\alpha_a\overline x}{d x^\alpha_a}}{\frac{d^\alpha_a x}{d x^\alpha_a}}
=1+\epsilon \frac{\tau^{(\alpha)}_a}{(x-a)^{1-\alpha}}.
\end{equation}
\end{definition}

We note that for $\alpha=1$ our Definition~\ref{invariantDef} coincides with
the standard approach (see, \textrm{e.g.}, \cite{Sarlet}). When $\Lambda\equiv 0$,
one obtains the concept of absolute invariance. The presence of a new function
$\Lambda$ is due to the presence of external forces in the dynamical system,
like friction. The function $\Lambda$ is called a \emph{Gauge} term.
In fact, many phenomena are nonconservative and this has to be taken into account
in the conservation laws \cite{Frederico:Torres3,Frederico:Torres4}. We give an example.

\begin{example}
Consider the transformation
\begin{equation}
\label{trans3}
\left\{
\begin{array}{l}
\overline x=x\\
\overline y= y+\epsilon \frac{1}{2\alpha}(x-a)^\alpha\\
\end{array}
\right.
\end{equation}
and the functional
\begin{equation}
\label{example}
\mathcal{J}(y)=\int_a^b\left(y^{(\alpha)}_a(x)\right)^2 \,d^\alpha_ax.
\end{equation}
Since
$$
\frac{d^\alpha_a}{dx^\alpha_a}\frac{1}{2\alpha}(x-a)^\alpha=\frac12,
$$
it is easy to verify that \eqref{example} is invariant under \eqref{trans3}
up to the Gauge function $\Lambda=y$.
\end{example}

\begin{definition}
\label{conservedDef}
Given a function $C=C(x,y,y^{(\alpha)}_a)$,
we say that $C$ is a conserved quantity for \eqref{funct} if
\begin{equation}
\label{conserved}
\frac{d^\alpha_a C}{dx^\alpha_a}(x,y(x),y^{(\alpha)}_a(x))=0
\end{equation}
along any solution $y$ of \eqref{FracELEquation}
(\textrm{i.e.}, along any extremal of \eqref{funct}).
\end{definition}

\begin{remark}
Applying the conformable integral \eqref{a5} to both sides
of equation \eqref{conserved}, Definition~\ref{conservedDef}
is equivalent to
$C(x,y(x),y^{(\alpha)}_a(x)) \equiv \, \mbox{const}$.
\end{remark}

We now provide a necessary condition of invariance.

\begin{theorem}
\label{TeoInv}
If $\mathcal{J}$ given by \eqref{funct} is invariant
under a family of transformations \eqref{trans}, then
\begin{equation}
\label{invar2}
\frac{\partial L}{\partial x}\tau + \frac{\partial L}{\partial y}\xi
+ \frac{\partial L}{\partial y^{(\alpha)}_a}\left[ \xi^{(\alpha)}_a
-y^{(\alpha)}_a\left((\alpha-1)\frac{\tau}{(x-a)}
+\frac{\tau^{(\alpha)}_a}{(x-a)^{1-\alpha}} \right) \right]
+L \frac{\tau^{(\alpha)}_a}{(x-a)^{1-\alpha}}=\frac{d^\alpha_a\Lambda}{dx^\alpha_a}.
\end{equation}
\end{theorem}

\begin{proof}
By the fractional chain rule (see Proposition~\ref{Ra2}),
$$
\frac{d^\alpha_a\overline y}{d\overline x^\alpha_a}
= \frac{\frac{d^\alpha_a\overline y}{dx^\alpha_a}}{(\overline x-a)^{\alpha-1}
\frac{d^\alpha_a\overline x}{dx^\alpha_a}}
=\frac{y^{(\alpha)}_a
+\epsilon \xi^{(\alpha)}_a }{ (x+\epsilon\tau-a)^{\alpha-1}[(x-a)^{1-\alpha}
+\epsilon\tau^{(\alpha)}_a]}.
$$
Substituting this formula into \eqref{invar}, differentiating with respect
to $\epsilon$ and then putting $\epsilon=0$, we obtain relation \eqref{invar2}.
\end{proof}

\begin{remark}
Allowing $\alpha$ to be equal to $1$, for $\Lambda\equiv0$ our equation \eqref{invar2}
becomes the standard necessary condition of invariance
(\textrm{cf.}, \textrm{e.g.}, \cite{Logan}):
$$
\frac{\partial L}{\partial x}\tau+  \frac{\partial L}{\partial y}\xi
+ \frac{\partial L}{\partial y'}( \xi'-y' \tau')+L \tau'=0.
$$
For $\alpha=1$ and an arbitrary $\Lambda$, see \cite{Sarlet}.
\end{remark}

In particular, if we consider ``time invariance''
(\textrm{i.e.}, $\tau\equiv0$), we obtain the following result.

\begin{corollary}
Let $\overline y=y+\epsilon \xi (x,y(x))$ be a transformation
that leaves invariant $\mathcal{J}$ in the sense that
$$
L(x,\overline y,\overline y^{(\alpha)}_a)
=L(x,y,y^{(\alpha)}_a)+\epsilon \frac{d^\alpha_a\Lambda}{dx^\alpha_a}(x,y)
+o(\epsilon).
$$
Then,
$$
\frac{\partial L}{\partial y}\xi
+ \frac{\partial L}{\partial y^{(\alpha)}_a}\xi^{(\alpha)}_a
=\frac{d^\alpha_a\Lambda}{dx^\alpha_a}.
$$
\end{corollary}


\section{The conformable fractional Noether theorem}
\label{sec:Noether}

Noether's theorem is a beautiful result with important implications
and applications in optimal control \cite{Torres,Torres1,Torres3}.
We provide here a conformable fractional Noether theorem
in the context of the calculus of variations. Later,
in Section~\ref{sec:FOC}, we provide a conformable fractional
optimal control version (see Theorem~\ref{fracNoether}).

\begin{theorem}[The conformable fractional Noether theorem]
If $\mathcal{J}$ given by \eqref{funct} is invariant under \eqref{trans}
and if $y$ is an extremal of $\mathcal{J}$, then
\begin{multline}
\label{invar20}
\frac{d^\alpha_a}{dx^\alpha_a} \left[ \left(L
-\frac{\partial L}{\partial y^{(\alpha)}_a}y^{(\alpha)}_a\right)\tau
+ \frac{\partial L}{\partial y^{(\alpha)}_a}\xi (x-a)^{1-\alpha}\right]\\
=(1-\alpha)\frac{\partial L}{\partial y^{(\alpha)}_a}\left[\xi (x-a)^{1-2\alpha}
-\frac{y^{(\alpha)}_a\tau}{(x-a)^{\alpha}}\right]
+\frac{d^\alpha_a\Lambda}{dx^\alpha_a}(x-a)^{1-\alpha}.
\end{multline}
\end{theorem}

\begin{proof}
From Theorem~\ref{TeoInv}, and using the conformable fractional Euler--Lagrange equation
\eqref{FracELEquation} and the DuBois--Reymond condition \eqref{duboiscondition},
we deduce successively that
$$
\begin{array}{ll}
\displaystyle \frac{d^\alpha_a\Lambda}{dx^\alpha_a} &(x-a)^{1-\alpha}\\
& \displaystyle =\left[  \frac{d^\alpha_a}{dx^\alpha_a}\left(L-\frac{\partial L}{\partial
y^{(\alpha)}_a}y^{(\alpha)}_a\right)\frac{\tau}{(x-a)^{1-\alpha}}
+ \frac{d^\alpha_a}{dx^\alpha_a}\left( \frac{\partial L}{\partial y^{(\alpha)}_a}\right)\xi
+ \frac{\partial L}{\partial y^{(\alpha)}_a}\xi^{(\alpha)}_a\right](x-a)^{1-\alpha}\\
&\displaystyle \qquad -\frac{\partial L}{\partial y^{(\alpha)}_a}y^{(\alpha)}_a\left[
\frac{(\alpha-1)\tau}{(x-a)^\alpha}+\tau^{(\alpha)}_a ) \right] +L \tau^{(\alpha)}_a\\
&\displaystyle = \left[  \frac{d^\alpha_a}{dx^\alpha_a}\left(L
-\frac{\partial L}{\partial y^{(\alpha)}_a}y^{(\alpha)}_a\right)\tau
+ \frac{d^\alpha_a}{dx^\alpha_a}\left( \frac{\partial L}{\partial
y^{(\alpha)}_a}\xi \right)(x-a)^{1-\alpha}\right]\\
&\displaystyle \qquad -\frac{\partial L}{\partial y^{(\alpha)}_a}y^{(\alpha)}_a\left[
\frac{(\alpha-1)\tau}{(x-a)^{\alpha}}+\tau^{(\alpha)}_a) \right] +L \tau^{(\alpha)}_a\\
&\displaystyle =  \frac{d^\alpha_a}{dx^\alpha_a} \left[ \left(L
-\frac{\partial L}{\partial y^{(\alpha)}_a}y^{(\alpha)}_a\right)\tau
+ \frac{\partial L}{\partial y^{(\alpha)}_a}\xi (x-a)^{1-\alpha}\right]
-\frac{\partial L}{\partial y^{(\alpha)}_a}y^{(\alpha)}_a\left[
\frac{(\alpha-1)\tau}{(x-a)^{\alpha}}+\tau^{(\alpha)}_a  \right] \\
&\displaystyle  \qquad + L \tau^{(\alpha)}_a-\left(L-\frac{\partial L}{\partial
y^{(\alpha)}_a}y^{(\alpha)}_a\right)\tau^{(\alpha)}_a-\frac{\partial L}{\partial
y^{(\alpha)}_a}\xi(1-\alpha)(x-a)^{1-2\alpha}\\
&\displaystyle =  \frac{d^\alpha_a}{dx^\alpha_a} \left[ \left(L-\frac{\partial L}{\partial
y^{(\alpha)}_a}y^{(\alpha)}_a\right)\tau+ \frac{\partial L}{\partial y^{(\alpha)}_a}\xi
(x-a)^{1-\alpha}\right]\\
&\displaystyle \qquad +\frac{\partial L}{\partial y^{(\alpha)}_a}y^{(\alpha)}_a
\frac{(1-\alpha)\tau}{(x-a)^{\alpha}}-\frac{\partial L}{\partial
y^{(\alpha)}_a}\xi (1-\alpha)(x-a)^{1-2\alpha}.
\end{array}
$$
Thus, we obtain equation \eqref{invar20}.
\end{proof}

\begin{remark}
When $\alpha=1$, equation \eqref{invar20} is simply Noether's
conservation law in the presence of external forces:
for any extremal of $\mathcal{J}$ and for any family
of transformations $(\overline x,\overline y)$
for which $\mathcal{J}$ is invariant, the conservation law
$$
\left(L-\frac{\partial L}{\partial y'}y'\right)\tau
+ \frac{\partial L}{\partial y'}\xi =\Lambda+ constant
$$
holds (see \cite[Theorem~2.1]{Sarlet}).
In addition, if system is conservative ($\Lambda\equiv 0$),
then one has the classical Noether theorem
$$
\left(L-\frac{\partial L}{\partial y'}y'\right)\tau
+ \frac{\partial L}{\partial y'}\xi =constant.
$$
\end{remark}

\begin{corollary}[The conformable fractional Noether theorem
under the presence of an external force $f$]
\label{corollaryNoether}
If $\mathcal{J}$ given by \eqref{funct} is invariant under \eqref{trans},
$y$ is an extremal of $\mathcal{J}$, and the function $f=f(x,y,y^{(\alpha)}_a)$
satisfies the equation
$$
\frac{d^\alpha_a f}{dx^\alpha_a}
=(1-\alpha)\frac{\partial L}{\partial y^{(\alpha)}_a}\left[\xi (x-a)^{1-2\alpha}
-\frac{y^{(\alpha)}_a\tau}{(x-a)^{\alpha}}\right]
+\frac{d^\alpha_a\Lambda}{dx^\alpha_a}(x-a)^{1-\alpha},
$$
then
$$
\left(L-\frac{\partial L}{\partial y^{(\alpha)}_a}y^{(\alpha)}_a\right)\tau
+ \frac{\partial L}{\partial y^{(\alpha)}_a}\xi (x-a)^{1-\alpha}-f
$$
is a conserved quantity.
\end{corollary}

\begin{corollary}
If $\mathcal{J}$ given by \eqref{funct} is invariant 
under the transformation $\overline x=x$, 
$\overline y=y+\epsilon \xi (x,y(x))$,
and if $y$ is an extremal of $\mathcal{J}$, then
$$
\frac{\partial L}{\partial y^{(\alpha)}_a}\xi -\Lambda
$$
is a conserved quantity.
\end{corollary}

\begin{proof}
The result is due to the fact that
$\frac{d^\alpha_a (x-a)^{1-\alpha}}{dx^\alpha_a}
=(1-\alpha)(x-a)^{1-2\alpha}$.
\end{proof}


\section{The Hamiltonian formalism}
\label{sec:Hamilton}

The Hamiltonian formalism is related to the Lagrangian one
by the so called Legendre transformation, from coordinates
and velocities to coordinates and momenta. Let momenta be given by
\begin{equation}
\label{defP}
p(x)=\frac{\partial L}{\partial y^{(\alpha)}_a}(x,y(x),y^{(\alpha)}_a(x))
\end{equation}
and the Hamiltonian function by
\begin{equation}
\label{defH}
H(x,y,v,\psi)=-L(x,y,v)+ \psi \, v.
\end{equation}
To simplify notation, $[y](x)$ and $\{y\}(x)$ will denote
$(x,y(x),y^{(\alpha)}_a(x))$ and $(x,y(x),y^{(\alpha)}_a(x),p(x))$, respectively.
Differentiating \eqref{defH}, and using definition \eqref{defP}, it follows that
\begin{equation}
\label{eq:cdH}
\begin{split}
\frac{d^\alpha_a H}{d x^\alpha_a}\{y\}(x)
&= -\frac{\partial L}{\partial x}[y](x)x^{(\alpha)}_a
-\frac{\partial L}{\partial y}[y](x) \cdot y^{(\alpha)}_a(x)
-\frac{\partial L}{\partial v}[y](x)
\cdot \frac{d^\alpha_a}{d x^\alpha_a}y^{(\alpha)}_a(x)\\
&\qquad +p^{(\alpha)}_a(x) \cdot y^{(\alpha)}_a(x)
+ \frac{\partial L}{\partial v}[y](x)
\cdot \frac{d^\alpha_a }{d x^\alpha_a}y^{(\alpha)}_a(x)\\
&=-\frac{\partial L}{\partial x}[y](x)\cdot (x-a)^{1-\alpha}
-\frac{\partial L}{\partial y}[y](x) \cdot y^{(\alpha)}_a(x)
+p^{(\alpha)}_a (x)\cdot y^{(\alpha)}_a(x).
\end{split}
\end{equation}
On the other hand, by the definition of Hamiltonian \eqref{defH},
one has immediately that
$$
\begin{cases}
\frac{\partial H}{\partial x}(x,y,v,\psi)
=-\frac{\partial L}{\partial x}(x,y,v)\\
\frac{\partial H}{\partial y}(x,y,v,\psi)
=-\frac{\partial L}{\partial y}(x,y,v)\\
\frac{\partial H}{\partial \psi}(x,y,v,\psi)=v
\end{cases}
$$
and so we can write equation \eqref{eq:cdH} in the form
\begin{equation}
\label{eq:bef:DR:HF}
\frac{d^\alpha_a H}{d x^\alpha_a}\{y\}(x)
= \frac{\partial H}{\partial x}\{y\}(x)(x-a)^{1-\alpha}
+\frac{\partial H}{\partial y}\{y\}(x) \cdot y^{(\alpha)}_a(x)
+\frac{\partial H}{\partial \psi}\{y\}(x) \cdot p^{(\alpha)}_a(x).
\end{equation}
If $y$ is an extremal of $\mathcal{J}$, then by the conformable fractional
Euler--Lagrange equation \eqref{FracELEquation}
$$
\frac{\partial L}{\partial y}[y](x)
-\frac{d^\alpha_a}{d x^\alpha_a}\left(\frac{\partial L}{\partial v}[y]\right)(x)
=-\frac{\partial H}{\partial y}\{y\}(x)-p^{(\alpha)}_a(x) = 0
$$
and we can write
\begin{equation}
\label{eq:EL:HF}
\left\{ \begin{array}{l}
\displaystyle y^{(\alpha)}_a(x)=\frac{\partial H}{\partial \psi}\{y\}(x),\\
\\
\displaystyle p^{(\alpha)}_a(x) =-\frac{\partial H}{\partial y}\{y\}(x).
\end{array} \right.
\end{equation}
The system \eqref{eq:EL:HF} is nothing else than the conformable
fractional Euler--Lagrange equation in Hamiltonian form.
Substituting the expressions of \eqref{eq:EL:HF}
into equation \eqref{eq:bef:DR:HF}, we get the analog to the DuBois--Reymond
condition \eqref{duboiscondition} in Hamiltonian form:
\begin{equation}
\label{eq:DR:Ham}
\frac{d^\alpha_a H}{d x^{\alpha}_a}\{y\}(x)
= \frac{\partial H}{\partial x}\{y\}(x)(x-a)^{1-\alpha}.
\end{equation}

If the Lagrangian $L$ is autonomous, \textrm{i.e.},
$L$ does not depend on $x$, then $\frac{\partial L}{\partial x} = 0$
and, consequently, by equation \eqref{eq:DR:Ham} $H$ is a conserved quantity.
If the Lagrangian $L$ does not depend on $y$, then
$\frac{\partial L}{\partial y}=- \frac{\partial H}{\partial y}= 0$
and so $p^{(\alpha)}_a=0$, \textrm{i.e.}, $p$ is a conserved quantity.
We now exhibit Corollary~\ref{corollaryNoether}
within the Hamiltonian framework.

\begin{theorem}[Conformable fractional Noether's theorem in Hamiltonian
form under the presence of an external force $f$]
If $\mathcal{J}$ given by \eqref{funct} is invariant under \eqref{trans},
$y$ is an extremal of $\mathcal{J}$, and function
$f=f(x,y(x),y^{(\alpha)}_a(x))$ satisfies the equation
$$
\frac{d^\alpha_a f}{dx^\alpha_a}(x,y(x),y^{(\alpha)}_a(x))
=(1-\alpha)p(x)\left[\xi (x-a)^{1-2\alpha}
-\frac{y^{(\alpha)}_a(x)\tau}{(x-a)^\alpha}\right]
+\frac{d^\alpha_a\Lambda}{dx^\alpha_a}(x,y(x))(x-a)^{1-\alpha},
$$
then
$$
p(x)\xi (x-a)^{1-\alpha}-H\{y\}(x)\tau-f(x,y(x),y^{(\alpha)}_a(x))
$$
is a conserved quantity.
\end{theorem}


\section{Conformable fractional optimal control}
\label{sec:FOC}

The conformable fractional optimal control problem is stated as follows:
find a pair of functions $(y(\cdot),v(\cdot))$ that minimizes
\begin{equation}
\label{functOptimal}
\mathcal{J}(y,v)=\int_a^b L(x,y(x),v(x)) \,d^\alpha_ax
\end{equation}
when subject to the (nonautonomous) fractional control system
\begin{equation}
\label{contraintOptimal}
y^{(\alpha)}_a(x)=\varphi (x,y(x),v(x)).
\end{equation}
A pair $(y(\cdot),v(\cdot))$ that minimizes functional \eqref{functOptimal}
subject to \eqref{contraintOptimal} is called an optimal process.
The reader interested on the fractional optimal control theory is referred to
\cite{Frederico:Torres2,Frederico:Torres5,MyID:259}.
Here we note that if $\alpha=1$, then \eqref{functOptimal}--\eqref{contraintOptimal}
is the standard optimal control problem: to minimize
$$
\mathcal{J}(y,v)=\int_a^b L(x,y(x),v(x)) \,dx
$$
subject to the control system
$$
y'(x)=\varphi (x,y(x),v(x)).
$$
We assume that the Lagrangian $L$ and the velocity vector $\varphi$ are functions
at least of class $C^1$ in their domain $[a,b]\times\mathbb{R}^2$.
Also, the admissible state trajectories $y$ are such that $y^{(\alpha)}_a$ exist.

\begin{remark}
In case $\varphi \equiv v$, the previous problem
\eqref{functOptimal}--\eqref{contraintOptimal}
reduces to the fundamental problem of the conformable
fractional variational calculus \eqref{funct},
as stated in Section \ref{sec:EL}.
\end{remark}

Following the standard approach \cite{CD:Djukic:1972,Torres3},
we consider the augmented conformable fractional functional
\begin{equation}
\label{augmfunction}
\mathcal{I}(y,v,p)
=\int_a^b [ L(x,y(x),v(x))+p(x)(y^{(\alpha)}_a(x)
-\varphi (x,y(x),v(x)))] \,d^\alpha_a x,
\end{equation}
where $p$ is such that $p^{(\alpha)}_a$ exists.
Consider a variation vector of type
$(y+\epsilon y_1,v+\epsilon v_1,p+\epsilon p_1)$
with $|\epsilon| \ll 1$. For convenience,
we restrict ourselves to the case $y_1(a)=y_1(b)=0$.
If $(y(\cdot),v(\cdot))$ is an optimal process,
then the first variation is zero when $\epsilon=0$.
Thus, using the conformable fractional integration
by parts formula (Theorem \ref{Ta3}), we obtain that
$$
\begin{array}{ll}
0&=\displaystyle\int_a^b \left[ \frac{\partial L}{\partial y}y_1
+ \frac{\partial L}{\partial v}v_1 +p_1 (y^{(\alpha)}_a-\varphi)
+p\left({y_1}^{(\alpha)}_a-\frac{\partial\varphi}{\partial y}y_1
- \frac{\partial\varphi}{\partial v}v_1\right) \right] \,d^\alpha_a x\\
&=\displaystyle\int_a^b \left[ y_1 \left( \frac{\partial L}{\partial y}
-p \frac{\partial\varphi}{\partial y}- p^{(\alpha)}_a \right)
+v_1\left( \frac{\partial L}{\partial v}
-p \frac{\partial\varphi}{\partial v} \right)
+ p_1 (y^{(\alpha)}_a-\varphi) \right] \,d^\alpha_a x.
\end{array}
$$
By the arbitrariness of the the variation functions,
we obtain the following system, called the Euler--Lagrange
equations for the conformable fractional optimal control problem:
\begin{equation}
\label{optimalsystem}
\left\{\begin{array}{l}
y^{(\alpha)}_a(x)=\varphi(x,y(x),v(x)),\\
p^{(\alpha)}_a(x)=\displaystyle\frac{\partial L}{\partial y}(x,y(x),v(x))
-p(x) \frac{\partial\varphi}{\partial y}(x,y(x),v(x)),\\
\displaystyle\frac{\partial L}{\partial v}(x,y(x),v(x))
-p (x)\frac{\partial\varphi}{\partial v}(x,y(x),v(x)) =0.
\end{array}\right.
\end{equation}
These equations give necessary conditions for finding the optimal solutions
of problem \eqref{functOptimal}--\eqref{contraintOptimal}.
We remark that they are similar to the standard ones, in case of
integer order derivatives, but in this case they
contain conformable fractional derivatives, as expected. The solution can be
stated using the Hamiltonian formalism. Consider the Hamiltonian function
\begin{equation}
\label{hamilDef}
H(x,y,v,p)=-L(x,y,v)+p(x) \varphi(x,y,v).
\end{equation}
Then \eqref{optimalsystem} gives:
\begin{enumerate}
\item the fractional Hamiltonian system
\begin{equation}
\label{Hamil1}
\left\{ \begin{array}{l}
\displaystyle y^{(\alpha)}_a(x)=\frac{\partial H}{\partial p}(x,y,v,p),\\[0.3cm]
\displaystyle p^{(\alpha)}_a(x)=-\frac{\partial H}{\partial y}(x,y,v,p);\\
\end{array} \right.
\end{equation}
\item the stationary condition
\begin{equation}
\label{Hamil2}
\frac{\partial H}{\partial v}(x,y,v,p)=0.
\end{equation}
\end{enumerate}

\begin{definition}
Any triplet $(y,v,p)$ satisfying system \eqref{Hamil1} and equation
\eqref{Hamil2} is called a conformable fractional Pontryagin extremal.
\end{definition}

\begin{remark}
In the particular case $\varphi \equiv v$, \textrm{i.e.},
when the conformable fractional optimal control problem is reduced
to the fundamental conformable fractional problem of the calculus of variations,
we obtain
$$
H=-L(x,y,v)+pv \, , \quad y^{(\alpha)}_a=v \, ,
$$
and the equations
$$
p^{(\alpha)}_a=-\frac{\partial H}{\partial y}
=\frac{\partial L}{\partial y}\, , \quad p=\frac{\partial L}{\partial v}.
$$
Therefore, we obtain the conformable
fractional Euler--Lagrange equation \eqref{FracELEquation}:
$$
\frac{\partial L}{\partial y}=\frac{d^\alpha_a}{dx^\alpha_a}\left(
\frac{\partial L}{\partial y^{(\alpha)}_a}\right).
$$
\end{remark}
Let us now considerer the augmented fractional variational functional
\eqref{augmfunction} written in the Hamiltonian form:
\begin{equation}
\label{fafh}
\mathcal{I}(y,v,p)=\int_0^1 (-H(x,y(x),v(x),p(x))
+p(x)y^{(\alpha)}_a(x)) \,d^\alpha_a x,
\end{equation}
where $H$ is given by expression \eqref{hamilDef}.
For a parameter $\epsilon$, with $|\epsilon| \ll 1$,
consider the family of transformations
\begin{equation}
\label{trans4}
\left\{
\begin{array}{l}
\overline x=x+\epsilon \tau (x,y(x),v(x),p(x)),\\
\overline y=y+\epsilon \xi (x,y(x),v(x),p(x)),\\
\overline v=v+\epsilon \sigma (x,y(x),v(x),p(x)),\\
\overline p=p+\epsilon \pi (x,y(x),v(x),p(x)).
\end{array}\right.
\end{equation}

We now define the notion of invariance of
\eqref{functOptimal}--\eqref{contraintOptimal}
in terms of the Hamiltonian $H$ and the augmented conformable
fractional variational functional \eqref{fafh}.

\begin{definition}
\label{def:inv}
The conformable fractional optimal control problem
\eqref{functOptimal}--\eqref{contraintOptimal}
is invariant under the transformations \eqref{trans4}
up to the Gauge term $\Lambda$, if a function $\Lambda=\Lambda(x,y)$
exists such that for any functions $y,v$ and $p$, and for any real
$x\in[0,1]$, the following equality holds:
\begin{equation}
\label{invar3}
\left[-H\left(\overline x,\overline y,\overline v,\overline p\right)
+\overline p \frac{d^\alpha_a\overline y}{d\overline x^\alpha_a}\right]
\frac{d^\alpha_a\overline x}{d^\alpha_a x}
=-H(x,y,v,p)+py^{(\alpha)}_a+\epsilon
\frac{d^\alpha_a\Lambda}{dx^\alpha_a}(x,y)+ o(\epsilon)
\end{equation}
for all $\epsilon$ in some neighborhood of zero,
where as in Definition~\ref{invariantDef}
$\frac{d^\alpha_a\overline x}{d^\alpha_a x}$
stands for \eqref{eq:standsfor}.
\end{definition}

\begin{theorem}[Fractional Noether's theorem for the fractional
optimal control problem \eqref{functOptimal}--\eqref{contraintOptimal}]
\label{fracNoether}
If \eqref{functOptimal}--\eqref{contraintOptimal} is invariant
under \eqref{trans4} in the sense of Definition~\ref{def:inv},
and if $(y,v,p)$ is a conformable fractional Pontryagin extremal, then
\begin{equation}
\label{invar4}
\frac{d^\alpha_a}{dx^\alpha_a}(p\xi)-\tau\left(\frac{\partial H}{\partial x}
+(\alpha-1)\frac{p y^{(\alpha)}_a}{x-a} \right)
-H\frac{\tau^{(\alpha)}_a}{(x-a)^{1-\alpha}}
=\frac{d^\alpha_a\Lambda}{dx^\alpha_a}.
\end{equation}
\end{theorem}

\begin{proof}
Differentiating \eqref{invar3} with respect to $\epsilon$,
and choosing $\epsilon=0$, we get
\begin{multline*}
-\frac{\partial H}{\partial x}\tau-\frac{\partial H}{\partial y}\xi
-\frac{\partial H}{\partial v}\sigma
-\frac{\partial H}{\partial p}\pi+\pi y^{(\alpha)}_a\\
+p\left[ \xi^{(\alpha)}_a-y^{(\alpha)}_a\left((\alpha-1)\frac{\tau}{x-a}
+\frac{\tau^{(\alpha)}_a}{(x-a)^{1-\alpha}}\right) \right]
+\left[-H+py^{(\alpha)}_a\right]\frac{\tau^{(\alpha)}_a}{(x-a)^{1-\alpha}}
=\frac{d^\alpha_a\Lambda}{dx^\alpha_a}.
\end{multline*}
Equation \eqref{invar4} follows because $(y,v,p)$
is a conformable fractional Pontryagin extremal.
\end{proof}

\begin{remark}
When $\alpha=1$ and $\Lambda=0$, equation \eqref{invar4} becomes
$$
\frac{d}{dx}(p\xi)-\tau\frac{\partial H}{\partial x}-H\tau'=0.
$$
Using relations \eqref{Hamil1} and \eqref{Hamil2}
with $\alpha=1$, we deduce that
$$
-H\tau+p\xi \equiv \, \mbox{constant},
$$
which is the optimal control version of Noether's theorem
\cite{Torres,Torres1,Torres3}. For $\alpha\in(0,1)$,
Theorem~\ref{fracNoether} extends
the main result of \cite{Frederico:Torres2}.
\end{remark}


\section{The multi-dimensional case}
\label{sec:MultiDim}

In this section we show a necessary condition of invariance, when the
Lagrangian depends on two independent variables $x_1$ and $x_2$ and on $m$
functions $y_1,\ldots,y_m$. First, we define conformable fractional partial
derivatives and conformable multiple fractional integrals in a natural way,
similarly as done in the integer case. In addition,
we are going to use the following properties.

\begin{theorem}[The conformable Green's theorem for a rectangle]
\label{greent}
Let $f$ and $g$ be two continuous and $\alpha$-differentiable functions
whose domains contain $R=[a,b]\times[c,d]\subset \mathbb{R}^2$. Then
\begin{multline}
\label{green}
\int_a^b\left(f(x_1,c)-f(x_1,d)\right)d^\alpha_ax_1
+\int_c^d\left(g(b,x_2)-g(a,x_2)\right)d^\alpha_cx_2\\
=\int_R  \left(\frac{\partial ^\alpha_a }{\partial {x_1}^\alpha_a} g(x_1,x_2)
-\frac{\partial ^\alpha_c }{\partial {x_2}^\alpha_c} f(x_1,x_2)\right)
d^\alpha_a x_1d^\alpha_c x_2.
\end{multline}
\end{theorem}

\begin{proof}
By Theorem~\ref{Ta2}, we have
\begin{equation*}
\begin{split}
f(x_1,d)-f(x_1,c)
&=\int_c^d\frac{\partial ^\alpha_c }{\partial
{x_2}^\alpha_c} f(x_1,x_2)d^\alpha_cx_2,\\
g(b,x_2)-g(a,x_2)
&=\int_a^b\frac{\partial ^\alpha_a }{\partial {x_1}^\alpha_a}
g(x_1,x_2)d^\alpha_ax_1.
\end{split}
\end{equation*}
Therefore,
\begin{equation*}
\begin{split}
\int_a^b &\left(f(x_1,c)-f(x_1,d)\right)d^\alpha_ax_1
+\int_c^d\left(g(b,x_2)-g(a,x_2)\right)d^\alpha_cx_2\\
&=-\int_a^b\int_c^d\frac{\partial ^\alpha_c }{\partial {x_2}^\alpha_c}
f(x_1,x_2)d^\alpha_cx_2d^\alpha_ax_1+\int_c^d\int_a^b
\frac{\partial ^\alpha_a }{\partial {x_1}^\alpha_a}
g(x_1,x_2)d^\alpha_ax_1d^\alpha_cx_2\\
&=\int_R  \left(\frac{\partial ^\alpha_a }{\partial {x_1}^\alpha_a} g(x_1,x_2)
-\frac{\partial ^\alpha_c }{\partial {x_2}^\alpha_c}
f(x_1,x_2)\right)d^\alpha_a x_1d^\alpha_c x_2.
\end{split}
\end{equation*}
The proof is complete.
\end{proof}

\begin{remark}
From Definition~\ref{Da2} and Remark~\ref{Ra1}, it is easy to verify that 
for $C^1$ functions our fractional Green's theorem over a rectangular domain 
(Theorem~\ref{greent}) reduces to the conventional Green's identity for 
$\tilde{f}(x_1,x_2)=f(x_1,x_2)(x_1-a)^{\alpha-1}$
and $\tilde{g}(x_1,x_2)=g(x_1,x_2)(x_2-a)^{\alpha-1}$.
\end{remark}

\begin{lemma}
Let $F$, $G$ and $h$ be continuous and $\alpha$-differentiable functions
whose domains contain $R=[a,b]\times[c,d]$. If $h=0$
on the boundary $\partial R$ of $R$, then
\begin{multline}
\label{lemma1}
\int_R  \left(G(x_1,x_2)\frac{\partial ^\alpha_a }{\partial {x_1}^\alpha_a}
h(x_1,x_2)-F(x_1,x_2)\frac{\partial ^\alpha_c }{\partial {x_2}^\alpha_c}
h(x_1,x_2)\right)d^\alpha_a x_1d^\alpha_c x_2\\
=-\int_R\left(\frac{\partial ^\alpha_a }{\partial {x_1}^\alpha_a} G(x_1,x_2)
-\frac{\partial ^\alpha_c }{\partial {x_2}^\alpha_c} F(x_1,x_2) \right)
h(x_1,x_2)d^\alpha_a x_1d^\alpha_c x_2.
\end{multline}
\end{lemma}

\begin{proof}
By choosing $f=Fh$ and $g=Gh$ in Green's formula \eqref{green}, we obtain that
\begin{equation*}
\begin{split}
\int_a^b & \left(F(x_1,c)h(x_1,c)-F(x_1,d)h(x_1,d)\right)d^\alpha_ax_1
+\int_c^d\left(G(b,x_2)g(b,x_2)-G(a,x_2)h(a,x_2)\right)d^\alpha_cx_2 \\
&=\int_R  \left(\frac{\partial ^\alpha_a }{\partial {x_1}^\alpha_a} G(x_1,x_2)
-\frac{\partial ^\alpha_c }{\partial {x_2}^\alpha_c}
F(x_1,x_2)\right)h(x_1,x_2)d^\alpha_a x_1d^\alpha_c x_2\\
& \quad +\int_R  \left(G(x_1,x_2)\frac{\partial^\alpha_a }{\partial
{x_1}^\alpha_a} h(x_1,x_2) -F(x_1,x_2)\frac{\partial^\alpha_c}{\partial
{x_2}^\alpha_c} h(x_1,x_2)\right)d^\alpha_a x_1d^\alpha_c x_2.
\end{split}
\end{equation*}
Since $h=0$ on the boundary $\partial R$ of $R$, we have
\begin{multline*}
\int_R \left(G(x_1,x_2)\frac{\partial^\alpha_a}{\partial {x_1}^\alpha_a}
h(x_1,x_2) -F(x_1,x_2)\frac{\partial^\alpha_c}{\partial {x_2}^\alpha_c}
h(x_1,x_2)\right)d^\alpha_a x_1d^\alpha_c x_2\\
=-\int_R  \left(\frac{\partial^\alpha_a}{\partial {x_1}^\alpha_a} G(x_1,x_2)
-\frac{\partial ^\alpha_c }{\partial {x_2}^\alpha_c}
F(x_1,x_2)\right)h(x_1,x_2)d^\alpha_a x_1d^\alpha_c x_2.
\end{multline*}
The proof is complete.
\end{proof}

\begin{remark}
In the very recent and general paper \cite{MR3428118}, 
a vector calculus with deformed derivatives 
(as the conformable derivative) 
is formally introduced. We refer the reader to 
\cite{MR3428118} for a detailed discussion 
of a vector calculus with deformed derivatives
and more properties on the multi-dimensional
conformable calculus.
\end{remark}

Let us consider now the fractional variational integral
\begin{equation}
\label{funct2}
\mathcal{J}(y)
=\int_R L\left(x,y,\frac{\partial^\alpha_a y}{\partial
x^\alpha_a}\right)d^\alpha_a x,
\end{equation}
where for simplicity we choose $R=[a,b] \times [a,b]$,
and where $x=(x_1,x_2)$, $y=(y_1,\ldots,y_m)$,
$d^\alpha_a x=d^\alpha_a x_1d^\alpha_a x_2$, and
$$
\frac{\partial ^\alpha_a y}{\partial x^\alpha_a}
=\left( \frac{\partial ^\alpha_a y_1}{\partial {x_1}^\alpha_a},
\frac{\partial ^\alpha_a y_1}{\partial {x_2}^\alpha_a},\ldots,
\frac{\partial ^\alpha_a y_m}{\partial {x_1}^\alpha_a},
\frac{\partial ^\alpha_a y_m}{\partial {x_2}^\alpha_a} \right).
$$
We are assuming that $L=L(x_1,x_2,y_1,\ldots,y_m,v_{1,1},v_{1,2},
\ldots,v_{m,1},v_{m,2})$ is at least of class $C^1$,
that the domains of $y_k,\, k\in\{1,\ldots,m\}$ contain $R$,
and that all these partial conformable fractional derivatives exist.

\begin{theorem}[The multi-dimensional fractional Euler--Lagrange equation]
Let $y$ be an extremizer of \eqref{funct2} with $y|_{\partial R}=\psi(x_1,x_2)$
for some given function $\psi=(\psi_1,\ldots,\psi_m)$. Then,
the following equation holds:
\begin{equation}
\label{FracELEquation2}
\frac{\partial L}{\partial y_k}-\frac{\partial^\alpha_a}{\partial {x_1}^\alpha_a}\left(
\frac{\partial L}{\partial v_{k,1}}\right)-\frac{\partial^\alpha_a}{\partial
{x_2}^\alpha_a}\left( \frac{\partial L}{\partial v_{k,2}}\right)=0
\end{equation}
for all $k\in\{1,\ldots,m\}$.
\end{theorem}

\begin{proof}
Let $y^{\ast}=(y_1^{\ast},\ldots,y_m^{\ast})$ give an extremum to \eqref{funct2}.
We define $m$ families of functions
\begin{equation}
\label{proof2}
y_k(x_1,x_2)=y_k^{\ast}(x_1,x_2)+\epsilon\eta_k(x_1,x_2),
\end{equation}
where $k\in \{1,\ldots,m\}$, $\epsilon$ is a constant, and $\eta_k$
is an arbitrary $\alpha$-differentiable function satisfying the boundary
conditions $\eta_k|_{\partial R}=0$ (weak variations). From \eqref{proof2},
the boundary conditions $\eta_k|_{\partial R}=0$ and
$y_k|_{\partial R}=\psi_k(x_1,x_2)$, it follows that function $y_k$ is admissible.
Let the Lagrangian $L$ be $C^1$. Because $y^{\ast}$ is an extremizer of functional
$\mathcal{J}$, the Gateaux derivative $\delta \mathcal{J}(y^{\ast})$ needs
to be identically null. For the functional \eqref{funct2},
\begin{equation*}
\begin{split}
\delta \mathcal{J}(y^{\ast})
&=\lim_{\epsilon\rightarrow 0}
\frac{1}{\epsilon}\left( \int_R L\left(x,y,\frac{\partial^\alpha_a y}{\partial
x^\alpha_a}\right) \,d^\alpha_ax
-\int_R L\left(x,y^{\ast},\frac{\partial^\alpha_a y^{\ast}}{\partial
x^\alpha_a}\right) \,d^\alpha_ax\right)\\
&=\sum_{k=1}^m\int_R \left(\eta_k(x_1,x_2)\frac{\partial
L\left(x,y^{\ast},\frac{\partial^\alpha_a y^{\ast}}{\partial
x^\alpha_a}\right)}{\partial y_k^{\ast}}\right.\\
&\left. +\frac{\partial^\alpha_a}{\partial {x_1}^\alpha_a}\eta_k(x_1,x_2)
\frac{\partial L\left(x,y^{\ast},\frac{\partial^\alpha_a y^{\ast}}{\partial
x^\alpha_a}\right)}{\partial v_{k,1}} +\frac{\partial^\alpha_a}{\partial
{x_2}^\alpha_a}\eta_k(x_1,x_2)\frac{\partial L\left(x,y^{\ast},
\frac{\partial^\alpha_a y^{\ast}}{\partial x^\alpha_a}\right)}{\partial
v_{k,2}}\right)d^\alpha_ax=0.
\end{split}
\end{equation*}
Using \eqref{lemma1}, we get that
\begin{equation}
\label{proof4}
\sum_{k=1}^m\int_R \eta_k(x_1,x_2)\left(\frac{\partial L\left(x,y^{\ast},
\frac{\partial^\alpha_a y^{\ast}}{\partial x^\alpha_a}\right)}{\partial y_k^{\ast}}
-\frac{\partial^\alpha_a}{\partial {x_1}^\alpha_a}\frac{\partial L\left(x,y^{\ast},
\frac{\partial^\alpha_a y^{\ast}}{\partial x^\alpha_a}\right)}{\partial v_{k,1}}
-\frac{\partial^\alpha_a}{\partial {x_2}^\alpha_a}\frac{\partial
L\left(x,y^{\ast},\frac{\partial^\alpha_a y^{\ast}}{\partial
x^\alpha_a}\right)}{\partial v_{k,2}} \right)d^\alpha_ax=0
\end{equation}
since $\eta_k|_{\partial R}=0$. The fractional Euler--Lagrange equation
\eqref{FracELEquation2} follows from \eqref{proof4} by using the fundamental
lemma \ref{Flemma}.
\end{proof}

Let $\epsilon$ be a real, and consider the following family of transformations:
\begin{equation}
\label{trans2}
\left\{
\begin{array}{l}
\overline x_i=x_i+\epsilon \tau_i (x,y(x)),
\quad i \in \{1,2\},\\
\overline y_k=y_k+\epsilon \xi_k (x,y(x)),
\quad k \in \{1,\ldots,m\},
\end{array}
\right.
\end{equation}
where $\tau_i$ and $\xi_k$ are such that there exist
$\frac{\partial^\alpha_a\tau_i}{\partial {x_j}^\alpha_a}$
and $\frac{\partial ^\alpha_a\xi_k}{\partial {x_j}^\alpha_a}$
for all $i,j \in \{1,2\}$ and all $ k \in \{1,\ldots,m\}$.
Denote by $\displaystyle\left[
\frac{\partial^\alpha_a \overline x}{\partial^\alpha_a x} \right]$
the matrix
$$
\displaystyle\left[
\begin{array}{cc}
\displaystyle\frac{\partial^\alpha_a \overline x_1/\partial
{x_1}^\alpha_a}{\partial^\alpha_a x_1/\partial {x_1}^\alpha_a}
& \displaystyle\frac{\partial^\alpha_a \overline x_1/\partial
{x_2}^\alpha_a}{\partial^\alpha_a x_2/\partial {x_2}^\alpha_a}\\
\displaystyle\frac{\partial^\alpha_a \overline x_2/\partial {x_1}^\alpha_a}{
\partial^\alpha_a x_1/\partial {x_1}^\alpha_a}
& \displaystyle\frac{\partial^\alpha_a \overline x_2/\partial {x_2}^\alpha_a}{
\partial^\alpha_a x_2/\partial {x_2}^\alpha_a}
\end{array}
\right]
= \displaystyle\left[
\begin{array}{cc}
\displaystyle 1+\frac{\epsilon}{(x_1-a)^{1-\alpha}}
\frac{\partial^\alpha_a\tau_1}{\partial {x_1}^\alpha_a}
& \displaystyle\frac{\epsilon}{(x_2-a)^{1-\alpha}}
\frac{\partial^\alpha_a\tau_1}{\partial {x_2}^\alpha_a}\\
\displaystyle\frac{\epsilon}{(x_1-a)^{1-\alpha}}
\frac{\partial^\alpha_a\tau_2}{\partial {x_1}^\alpha_a}
&\displaystyle 1+\frac{\epsilon}{(x_2-a)^{1-\alpha}}
\frac{\partial^\alpha_a\tau_2}{\partial {x_2}^\alpha_a}
\end{array}
\right].
$$

\begin{definition}
Functional $\mathcal{J}$ as in \eqref{funct2} is invariant under the family of
transformation \eqref{trans2} if for all $y_k$ and for all $x_i\in[0,1]$ we have
$$
L\left(\overline x,\overline y,\frac{\partial^\alpha_a\overline y}{
\partial\overline x^\alpha_a}\right)
det \left[ \frac{\partial^\alpha_a \overline x}{\partial^\alpha_a x} \right]
=L\left(x,y,\frac{\partial^\alpha_a y}{\partial x^\alpha_a}\right)
+\epsilon \frac{d^\alpha_a\Lambda}{dx^\alpha_a}(x,y)+o(\epsilon)
$$
for all $\epsilon$ in some neighborhood of zero.
\end{definition}

Using the same techniques as in the proof of Theorem~\ref{TeoInv},
we obtain a necessary condition of invariance
for the fractional variational problem \eqref{funct2}.

\begin{theorem}
\label{TeoInv2}
If $\mathcal{J}$ given by \eqref{funct2} is invariant
under transformations \eqref{trans2}, then
\begin{multline}
\label{eq:cni:md}
\sum_{i=1}^2 \frac{\partial L}{\partial x_i}\tau_i
+ \sum_{k=1}^m \frac{\partial L}{\partial y_k}\xi_k
+ \sum_{k=1}^m \sum_{i=1}^2\frac{\partial L}{\partial v_{k,i}}\left[
\frac{\partial ^\alpha_a \xi_k}{\partial {x_i}^\alpha_a}
- \frac{\partial ^\alpha_a y_k}{\partial {x_i}^\alpha_a}\left(
(\alpha-1)\frac{\tau_i}{x_i-a}+\frac{1}{(x_i-a)^{1-\alpha}}
\frac{\partial^\alpha_a \tau_i }{\partial {x_i}^\alpha_a} \right) \right]\\
+ L\left( \frac{1}{(x_1-a)^{1-\alpha}}
\frac{\partial^\alpha_a\tau_1}{\partial {x_1}^\alpha_a}
+ \frac{1}{(x_2-a)^{1-\alpha}} \frac{\partial^\alpha_a\tau_2}{
\partial {x_2}^\alpha_a}\right)=\frac{d^\alpha\Lambda}{dx^\alpha}.
\end{multline}
\end{theorem}

\begin{proof}
Using relations
$$
\displaystyle\frac{\partial^\alpha_a\overline y_k}{
\partial{{\overline x}_i}^\alpha_a}
=\frac{\frac{\partial^\alpha_a y_k}{\partial {x_i}^\alpha_a}
+\epsilon \frac{\partial^\alpha_a \xi_k}{\partial {x_i}^\alpha_a} }{(x_i
+\epsilon\tau_i-a)^{\alpha-1}\left[ (x_i-a)^{1-\alpha}
+\epsilon \frac{\partial^\alpha_a \tau_i }{\partial {x_i}^\alpha_a}  \right]}
$$
and
$$
\displaystyle\left.\frac{d}{d\epsilon}det \left[
\frac{\partial^\alpha \overline x}{\partial^\alpha x} \right]\right|_{\epsilon=0}
=\frac{1}{(x_1-a)^{1-\alpha}} \frac{\partial^\alpha_a\tau_1}{\partial {x_1}^\alpha_a}
+ \frac{1}{(x_2-a)^{1-\alpha}} \frac{\partial^\alpha_a\tau_2}{\partial {x_2}^\alpha_a},
$$
we conclude that \eqref{eq:cni:md} holds.
\end{proof}

\begin{remark}
When $\alpha=1$ and $\Lambda\equiv0$, Theorem~\ref{TeoInv2}
reduces to the standard one (\textrm{cf.} \cite{Logan}):
equality \eqref{eq:cni:md} simplifies to
$$
\sum_{i=1}^2 \frac{\partial L}{\partial x_i}\tau_i
+ \sum_{k=1}^m \frac{\partial L}{\partial y_k}\xi_k
+ \sum_{k=1}^m \sum_{i=1}^2\frac{\partial L}{\partial v_{k,i}}\left[
\frac{\partial \xi_k}{\partial x_i}-\frac{\partial y_k}{\partial x_i}
\frac{\partial \tau_i }{\partial x_i}\right]
+L\left(\frac{\partial\tau_1}{\partial x_1}
+ \frac{\partial\tau_2}{\partial x_2}\right)=0.
$$
\end{remark}

\begin{corollary}
If $\mathcal{J}$ given by \eqref{funct2} is invariant under \eqref{trans2},
$\tau_1\equiv0\equiv\tau_2$, and no Gauge term is involved
(i.e., $\Lambda\equiv0$), then
$$
\sum_{k=1}^m \frac{\partial L}{\partial y_k}\xi_k
+  \sum_{k=1}^m \sum_{i=1}^2\frac{\partial L}{\partial
v_{k,i}}\frac{\partial ^\alpha_a \xi_k}{\partial {x_i}^\alpha_a}=0.
$$
\end{corollary}

It remains an open question how to obtain a Noether constant
of motion for the conformable fractional multi-dimensional case.


\section*{Acknowledgments}

This work was partially supported by CNPq and CAPES (Brazilian research funding
agencies), and by Portuguese funds through the \emph{Center for Research and
Development in Mathematics and Applications} (CIDMA), and \emph{The Portuguese 
Foundation for Science and Technology} (FCT), within project UID/MAT/04106/2013.
The authors are grateful to three anonymous Reviewers 
for useful comments and suggestions.




\begin{thebibliography}{99}

\bibitem{Abd:15}
T. Abdeljawad,
On conformable fractional calculus,
J. Comput. Appl. Math. {\bf 279} (2015), 57--66.

\bibitem{CD:Agrawal:2002} 
O. P. Agrawal, 
Formulation of Euler-Lagrange equations for fractional variational problems,
J. Math. Anal. Appl. {\bf 272} (2002), no.~1, 368--379.

\bibitem{book:frac:ICP2}
R. Almeida, S. Pooseh\ and\ D. F. M. Torres,
{\it Computational methods in the fractional calculus of variations},
Imp. Coll. Press, London, 2015.

\bibitem{AT}
R. Almeida, D. F. M. Torres, 
Necessary and sufficient conditions for the fractional calculus
of variations with Caputo derivatives, 
Commun. Nonlinear Sci. Numer. Simul. {\bf 16} (2011), 1490--1500.
{\tt arXiv:1007.2937}

\bibitem{And:Avery}
D. R. Anderson\ and\ R. Avery,
Fractional-order boundary value problem with Sturm-Liouville boundary conditions,
Electron. J. Differential Equations {\bf 2015} (2015), 10~pp.

\bibitem{MR3428118}
A. S. Balankin, J. Bory-Reyes\ and\ M. Shapiro, 
Towards a physics on fractals: differential vector calculus 
in three-dimensional continuum with fractal metric, 
Phys. A {\bf 444} (2016), 345--359.

\bibitem{BA} 
D. Baleanu, O. P. Agrawal, 
Fractional Hamilton formalism within Caputo's derivative,
Czechoslovak J. Phys. {\bf 56}, (2006), 1087--1092.

\bibitem{MR3326681}
H. Batarfi, J. Losada, J. J. Nieto\ and\ W. Shammakh,
Three-point boundary value problems for conformable fractional differential equations,
J. Funct. Spaces {\bf 2015} (2015), Art. ID 706383, 6~pp.

\bibitem{Bauer} 
P. S. Bauer, 
Dissipative Dynamical Systems: I, 
Proc. Natl. Acad. Sci {\bf 17} (1931), 311--314.

\bibitem{MyID:337}
B. Bayour\ and\ D. F. M. Torres,
Existence of solution to a local fractional nonlinear differential equation,
J. Comput. Appl. Math., in press.
DOI:10.1016/j.cam.2016.01.014
{\tt arXiv:1601.02126}

\bibitem{MyID:296}
N. Benkhettou, A. M. C. Brito da Cruz\ and\ D. F. M. Torres,
A fractional calculus on arbitrary time scales:
fractional differentiation and fractional integration,
Signal Process. 107 (2015), 230--237.
{\tt arXiv:1405.2813}

\bibitem{MyID:320}
N. Benkhettou, A. M. C. Brito da Cruz\ and\ D. F. M. Torres,
Nonsymmetric and symmetric fractional calculi on arbitrary nonempty closed sets,
Math. Methods Appl. Sci. {\bf 39} (2016), no.~2, 261--279.
{\tt arXiv:1502.07277}

\bibitem{MyID:324}
N. Benkhettou, S. Hassani\ and\ D. F. M. Torres,
A conformable fractional calculus on arbitrary time scales,
J. King Saud Univ. Sci. {\bf 28} (2016), no.~1, 93--98.
{\tt arXiv:1505.03134}

\bibitem{Cresson}
J. Cresson, 
Fractional embedding of differential operators and Lagrangian systems, 
J. Math. Phys. {\bf 48} (2007), no.~3, 033504, 34~pp.
{\tt arXiv:math/0605752}

\bibitem{CD:Djukic:1972}
\Dbar. S. \Dbar uki\'c,
Noether's theorem for optimum control systems,
Internat. J. Control (1) {\bf 18} (1973), 667--672.

\bibitem{Frederico:Torres1}
G. S. F. Frederico\ and\ D. F. M. Torres,
A formulation of Noether's theorem for fractional problems of the calculus of variations,
\textit{J. Math. Anal. Appl.} {\bf 334} (2007), no.~2, 834--846.
{\tt arXiv:math/0701187}

\bibitem{Frederico:Torres3}
G. S. F. Frederico\ and\ D. F. M. Torres,
Non-conservative Noether's theorem for fractional action-like
variational problems with intrinsic and observer times,
\textit{Int. J. Ecol. Econ. Stat.} {\bf 9} (2007), no.~F07, 74--82.
{\tt arXiv:0711.0645}

\bibitem{Frederico:Torres4}
G. S. F. Frederico\ and\ D. F. M. Torres,
Nonconservative Noether's theorem in optimal control,
\textit{Int. J. Tomogr. Stat.} {\bf 5} (2007), no.~W07, 109--114.
{\tt arXiv:math/0512468}

\bibitem{Frederico:Torres2}
G. S. F. Frederico\ and\ D. F. M. Torres,
Fractional conservation laws in optimal control theory,
Nonlinear Dynam. {\bf 53} (2008), no.~3, 215--222.
{\tt arXiv:0711.0609}

\bibitem{Frederico:Torres5}
G. S. F. Frederico\ and\ D. F. M. Torres,
Fractional optimal control in the sense of Caputo and the fractional Noether's theorem,
Int. Math. Forum {\bf 3} (2008), no.~10, 479--493.
{\tt arXiv:0712.1844}

\bibitem{Frederico:Torres6}
G. S. F. Frederico\ and\ D. F. M. Torres,
Fractional Noether's theorem in the Riesz-Caputo sense,
Appl. Math. Comput. {\bf 217} (2010), no.~3, 1023--1033.
{\tt arXiv:1001.4507}

\bibitem{MR3164103}
R. Khalil, M. Al Horani, A. Yousef\ and\ M. Sababheh,
A new definition of fractional derivative,
J. Comput. Appl. Math. {\bf 264} (2014), 65--70.

\bibitem{Kilbas}
A. A. Kilbas, H. M. Srivastava\ and\ J. J. Trujillo,
{\it Theory and applications of fractional differential equations},
Elsevier, Amsterdam, 2006.

\bibitem{LazoCesar}
M. J. Lazo\ and\ C. E. Krumreich,
The action principle for dissipative systems,
J. Math. Phys. {\bf 55} (2014), 122902, 11~pp.
{\tt arXiv:1412.5109}

\bibitem{LazoTorres1} 
M. J. Lazo\ and\ D. F. M. Torres, 
The DuBois-Reymond fundamental lemma of the fractional 
calculus of variations and an Euler-Lagrange equation 
involving only derivatives of Caputo, 
J. Optim. Theory Appl. {\bf 156} (2013), no.~1, 56--67.
{\tt arXiv:1210.0705}

\bibitem{Logan}
J. D. Logan,
{\it Invariant variational principles},
Mathematics in Science and Engineering,
Vol. 138. Academic Press, New York-London, 1977.

\bibitem{Machado}
J. T. Machado, V. Kiryakova\ and\ F. Mainardi,
Recent history of fractional calculus,
Commun. Nonlinear Sci. Numer. Simul. {\bf 16} (2011), no.~3, 1140--1153.

\bibitem{MR3331286}
A. B. Malinowska, T. Odzijewicz\ and\ D. F. M. Torres, 
{\it Advanced methods in the fractional calculus of variations}, 
Springer Briefs in Applied Sciences and Technology, Springer, Cham, 2015. 

\bibitem{book:frac}
A. B. Malinowska\ and\ D. F. M. Torres, 
{\it Introduction to the fractional calculus of variations}, 
Imp. Coll. Press, London, 2012. 

\bibitem{miller}
K. S. Miller\ and\ B. Ross,
{\it An introduction to the fractional calculus and fractional differential equations},
Wiley, New York, 1993.

\bibitem{OMT}
T. Odzijewicz, A. B. Malinowska\ and\ D. F. M. Torres, 
Fractional variational calculus with classical and combined Caputo derivatives, 
Nonlinear Anal. {\bf 75} (2012), no.~3, 1507--1515.
{\tt arXiv:1101.2932}

\bibitem{MyID:227}
T. Odzijewicz, A. B. Malinowska\ and\ D. F. M. Torres, 
Fractional calculus of variations in terms of a generalized fractional integral 
with applications to physics, 
Abstr. Appl. Anal. {\bf 2012} (2012), Art. ID 871912, 24~pp. 
{\tt arXiv:1203.1961}

\bibitem{Podlubny}
I. Podlubny,
{\it Fractional differential equations},
Academic Press, San Diego, CA, 1999.

\bibitem{MyID:259}
S. Pooseh, R. Almeida\ and\ D. F. M. Torres,
Fractional order optimal control problems with free terminal time,
J. Ind. Manag. Optim. 10 (2014), no.~2, 363--381.
{\tt arXiv:1302.1717}

\bibitem{CD:Riewe:1997}
F. Riewe,
Mechanics with fractional derivatives,
Phys. Rev. E (3) {\bf 55} (1997), no.~3 part B, 3581--3592.

\bibitem{Sarlet}
W. Sarlet, and F. Cantrijn,
Generalizations of Noether's theorem in classical mechanics,
SIAM Rev. {\bf 23} (1981), no.~4, 467--494.

\bibitem{Torres}
D. F. M. Torres,
On the Noether theorem for optimal control,
Eur. J. Control {\bf 8} (2002), no.~1, 56--63.

\bibitem{Torres1}
D. F. M. Torres,
Conservation laws in optimal control,
in \textit{Dynamics, bifurcations, and control},
Lecture Notes in Control and Inform. Sci. {\bf 273},
Springer, Berlin (2002), 287--296.

\bibitem{Torres3}
D. F. M. Torres,
Quasi-invariant optimal control problems,
Port. Math. {\bf 61} (2004), no.~1, 97--114.
{\tt arXiv:math/0302264}

\bibitem{Torres2}
D. F. M. Torres,
Proper extensions of Noether's symmetry theorem
for nonsmooth extremals of the calculus of variations,
Commun. Pure Appl. Anal. {\bf 3} (2004), no.~3, 491--500.

\end{thebibliography}
\end{document}